%% file: ms.tex
\documentclass[opre]{informs3Itai}

\OneAndAHalfSpacedXI 

\TheoremsNumberedThrough     
\ECRepeatTheorems

\EquationsNumberedThrough    

\renewenvironment{proof}[1][Proof]{\textsc{#1.} }{ \hfill $\square$ \medskip}

\let\argmax\relax
\usepackage[x11names]{xcolor}
\usepackage[colorlinks]{hyperref}
\hypersetup{citecolor=magenta}

\newcommand{\bsq}{\vrule height .9ex width .8ex depth -.1ex}
\renewenvironment{remark}{\begin{oldremark}}{%
    \hfill \bsq
    \end{oldremark}\ignorespacesafterend%
}
\renewenvironment{example}{\begin{oldexample}}{%
    \hfill \bsq
    \end{oldexample}\ignorespacesafterend%
}

\usepackage{amsmath,amsfonts,dsfont,amssymb}
\usepackage{tikz}
\usepackage{multirow,adjustbox,graphicx}  

\newcommand{\R}{\mathbb{R}}								
\newcommand{\N}{\mathbb{N}}								
\newcommand{\Rp}{\mathbb{R}_{\geq 0}}					
\newcommand{\E}{\mathbb{E}}								
\renewcommand{\Pr}{\mathbb{P}}							
\newcommand{\In}[1]{\mathds{1}_{\crl*{ #1}}} 			
\newcommand{\Ins}[1]{\mathds{1}_{#1}} 					
\newcommand{\defeq}{\coloneqq}							

\usepackage{mathtools}
\DeclarePairedDelimiter{\abs}{\lvert}{\rvert}
\DeclarePairedDelimiter{\norm}{\lvert\lvert}{\rvert\rvert}
\DeclarePairedDelimiter{\crl}{\{}{\}} 
\DeclarePairedDelimiter{\prn}{(}{)} 
\DeclarePairedDelimiter{\brk}{[}{]} 

\usepackage{xspace}
\newcommand{\T}{T}										
\newcommand{\voff}{v^{\text{off}}}						
\newcommand{\von}{v^{\text{on}}}						

\newcommand{\rmax}{r_{\varphi}}							
\newcommand{\wmax}{w_{\max}}
\newcommand{\onl}{\textsc{Online}\xspace}
\newcommand{\off}{\textsc{Offline}\xspace}
\DeclareMathOperator{\reg}{Regret}
\DeclareMathOperator{\Bin}{Bin}

\DeclareMathOperator{\argmax}{argmax}
\newcommand{\Xt}{X^t}									
\newcommand{\Xts}{X^{\star t}}							
	
\let\U\relax
\newcommand{\U}{\mathcal{U}}							
\let\u\relax
\newcommand{\u}{u}					
\newcommand{\Ut}{U^t}
\newcommand{\hUt}{\hat U^t}
\renewcommand{\S}{\mathcal{S}}							
\newcommand{\Tr}{\mathcal{T}}							
\renewcommand{\Re}{\mathcal{R}}                         
\newcommand{\xit}{\xi^t}
\newcommand{\xiT}{\xi^T}

\newcommand{\St}{S^t}
\newcommand{\Nt}{N^t}
\newcommand{\calB}{\mathcal{B}}							
\newcommand{\calF}{\mathcal{F}}
\newcommand{\calG}{\mathcal{G}}
\newcommand{\Rt}{R^t}
\newcommand{\bRt}{\bar R^t}
\newcommand{\sigmat}{\hat\sigma^t}
\newcommand{\Bp}{B_p}									
\newcommand{\Bh}{B_h}									
\newcommand{\Bt}{B^t}

\newcommand{\BL}{L_B}               
\newcommand{\f}{f}

\newcommand{\vf}{\vec{f}}
\newcommand{\vp}{\vec{p}}    
\newcommand{\vw}{\vec{w}}
\newcommand{\vr}{\vec{r}}
\newcommand{\vq}{\vec{q}}
\newcommand{\vx}{\vec{x}}
\newcommand{\vy}{\vec{y}}
\newcommand{\vb}{\vec{b}}
\newcommand{\vd}{\vec{d}}
\newcommand{\ve}{\vec{e}}
\newcommand{\vz}{\vec{z}}

\newcommand{\CP}{P_C} 
\newcommand{\SP}{P_S} 

\newcommand{\Alpha}{\mathcal{A}}

\newcommand{\accept}{\texttt{a}\xspace}
\newcommand{\reject}{\texttt{r}\xspace}
\newcommand{\probe}{\texttt{p}\xspace}
\newcommand{\Gt}{\calG_t}
\newcommand{\calQ}{\mathcal{Q}}
\newcommand{\hatphi}{\hat\varphi}

\newcommand{\oknap}{\texttt{OnlineKnapsack}\xspace}

\newcommand{\rabbi}{\textsc{rabbi}\xspace}
\newcommand{\bandits}{\textsc{Bandits}\xspace}
\newcommand{\explore}{\texttt{explore}\xspace}
\newcommand{\exploit}{\texttt{exploit}\xspace}

\renewcommand{\paragraph}[1]{\smallbreak\noindent \textbf{#1}}
\let\vec\mathbf

\usepackage{algorithm}
\usepackage[noend]{algpseudocode}

\newenvironment{proofof}[1]{%
	\begin{proof}[{\sc Proof of #1}]%
	}{%
	\end{proof}%
}

\usepackage[capitalize]{cleveref}

\usepackage{natbib}
\bibpunct[, ]{(}{)}{,}{a}{}{,}%
\def\BIBand{and}%
 
\begin{document}


 \RUNAUTHOR{Vera, Banerjee and Gurvich}

\RUNTITLE{Online allocation via Bellman Inequalities}

\TITLE{Online Allocation and Pricing: Constant Regret via Bellman Inequalities}

\ARTICLEAUTHORS{%
\AUTHOR{Alberto Vera \qquad Siddhartha Banerjee \qquad Itai Gurvich}
\AFF{School of Operations Research and Information Engineering, Cornell University\\ \EMAIL{aav39@cornell.edu \qquad sbanerjee@cornell.edu \qquad gurvich@cornell.edu}} 
} 

\ABSTRACT{%
We develop a framework for designing simple and efficient policies for a family of online allocation and pricing problems, that includes online packing, budget-constrained probing, dynamic pricing, and online contextual bandits with knapsacks.
In each case, we evaluate the performance of our policies in terms of their regret (i.e., additive gap) relative to an offline controller that is endowed with more information than the online controller. 
Our framework is based on Bellman Inequalities, which decompose the loss of an algorithm into two distinct sources of error: (1) arising from computational tractability issues, and (2) arising from estimation/prediction of random trajectories.
Balancing these errors guides the choice of benchmarks, and leads to policies that are both tractable and have strong performance guarantees. 
In particular, in all our examples, we demonstrate constant-regret policies that only require re-solving an LP in each period, followed by a simple greedy action-selection rule; thus, our policies are practical as well as provably near optimal.
}%


\maketitle
 \vspace*{-0.8cm}

\section{Introduction}
Online decision-making under uncertainty is widely studied across a variety of fields, including operations research, control, and computer science. 
A canonical framework for such problems is that of Markov decision processes (MDP), with associated use of stochastic dynamic programming for designing policies. 
In complex settings, however, such approaches suffer from the known curse-of-dimensionality; moreover, they also fail to provide insights into structural properties of the problem: the performance of heuristics, dependence on distributional information, etc.

The above challenges have inspired an alternate approach to designing approximate policies for MDPs based on the use of \emph{benchmarks} -- proxies for the value function that provide bounds for the optimal policy, and guide the design of heuristics.
The performance of any policy can be quantified by its additive loss, or \emph{regret}, relative to any such benchmark; this consequently also bounds the additive optimality gap, i.e., performance against the optimal policy.

In this work, we develop \emph{new policies for online resource-allocation problems}: settings where a finite set of resources is dynamically allocated to arriving requests, with associated constraints and rewards/costs. 
Our baseline problem is the online stochastic knapsack problem (henceforth \oknap):
a controller has initial inventory $B$, and requests arrive sequentially over horizon $T$. Each request has a random type corresponding to a resource requirement-reward pair. Requests are generated from a known stochastic process, and are revealed upon arrival; the controller must then decide whether to accept/reject each request, in order to maximize rewards while satisfying budget constraints.
We then consider three variants of this basic setting: (1) online probing, (2) dynamic pricing, and (3) contextual bandits with knapsacks.
These are widely-studied problems, each of which augments the baseline \oknap with additional constraints/controls.
The formal models for these settings are presented in \cref{sec:prelim}

Instead of solving each problem in an ad-hoc manner, however, our policies are all derived from a single underlying framework.
In particular, our results can be summarized as follows:
\smallskip
\paragraph{Meta-theorem} \emph{Given an online allocation problem, we identify an appropriate offline benchmark, and give a simple online policy --  based on solving a tractable optimization problem in each period -- that gets constant regret compared to the benchmark (and thus, compared to the optimal policy).} 
\smallskip

In more detail, our approach is based on adaptively constructing a benchmark that has additional (but not necessarily full) information about future randomness.
Next, in the spirit of online primal-dual methods, we  use our benchmark to construct a feasible online policy.
The centerpiece of our approach are the \emph{Bellman Inequalities}, which characterize what benchmarks are feasible, and also, decompose the regret of an online policy into two distinct terms.
The first, which we call the \emph{Bellman Loss}, arises from computational considerations, specifically, from requiring that the benchmark is tractable (instead of a dynamic program which may be intractable);
The second, which we call the \emph{Information Loss}, accounts for unpredictability across sample paths.
Our policies trade off these two losses to get strong performance guarantees.

Our framework allows flexibility in choosing benchmarks. 
To understand why this is important, consider two common benchmarks for dynamic pricing: a controller has inventory $B$, and posts prices for $T$ sequential customers, each of who has a random valuation.
One common benchmark, known as the \emph{offline} or \emph{prophet} benchmark, considers a controller with \emph{full information} of all randomness; it is easy to show that no online policy can get better than $\Omega(T)$ regret against this benchmark.. 
An alternate benchmark, known as the \emph{ex ante} or \emph{fluid} benchmark, corresponds to replacing all random quantities with their expectations; here again, no online policy can get better than $\Omega(\sqrt{T})$ regret~\citep{bayes_prophet}. Our approach however lets us identify benchmarks which have $O(1)$ regret for all our settings.

Prophet and fluid benchmarks are also widely used in adversarial models of online allocation, leading to algorithms with worst-case guarantees.
In contrast, we consider stochastic inputs, and consequently get much stronger guarantees.
In particular, all our guarantees are \emph{parametric} and depend explicitly on the distributions and problem primitives (i.e., constant parameters defining the instance). 
All our policies, however, have regret that is independent of the horizon and budgets.

\section{Preliminaries and Overview}

\label{sec:prelim}

\subsection{Problem Settings and Results} \label{sec:setting}

We illustrate our framework by developing low-regret algorithms for the following problems:

\paragraph{Online Stochastic Knapsack.}
This serves as a baseline for our other problems. The controller has an initial resource budget $B$, and items arrive sequentially over $T$ periods. Each item has a random {type} $j$ which corresponds to a \emph{known} resource requirement (or `weight') $w_j$ and a \emph{random} reward $R_j$.
In period $t=T,T-1,\ldots,1$ (where $t$ denotes the \emph{time-to-go}), we assume the arriving type is drawn from a finite set $[n]$ from some known distribution $\vp=(p_1,\ldots,p_n)$. 
At the start of each period, the controller observes the type of the arriving item, and must decide to accept or reject the item.
The expected reward from selecting a type-$j$ item is $r_j = \E[R_j]$.

\paragraph{Online Probing.} As before an arriving type $j$ has known expected reward $r_j$, but unknown realized reward $R_j$ -- now the controller has the additional option of probing each request to observe the realization, and then accept/reject the item based on the revealed reward; the controller can also choose to accept the item without probing. 
In addition to the resource budget $B$, the controller has an additional probing budget $B_p$ that limits the number of arrivals that can be probed.
This introduces a trade-off between depleting the resource budget $B$ and probing budget $B_p$.
We assume here that $R_j$ has finite support $\crl{r_{jk}}_{k\in[m]}$ of size $m$, and define $q_{jk}\defeq\Pr[R_j=r_{jk}]$ for $k\in[m]$. 
Note this reduces to \oknap when either $\Bp\geq \T$ or $\Bp=0$.

\paragraph{Dynamic Pricing.} The controller has an initial inventory $B\in\N^d$ for $d$ different resources.
There are $n$ types of customers, where a customer of type $j$ requests a specific subset $A_j\in\crl{0,1}^d$ of resources, and has private valuation $R^t\sim F_j$.
In each period $t$, the controller observes the customer type $j\in [n]$, and if sufficient resources are available, posts a price (fare) $\f$ from a finite set $\crl{\f_{j1},\ldots,\f_{jm}}$; the customer then purchases iff $R^t>\f$.
The vectors $A_j$ and valuation functions $(F_j:j\in [n])$ are known, but otherwise arbitrary.
More generally, our technique handles probabilistic customer-choice models, where a customer, when presented with a price menu over bundles, picks a random bundle via some known distribution (which may depend on the menu).

\paragraph{Knapsack with Distribution Learning.} 
We return to the \oknap setting where items of type $j\in[n]$ have weight $w_j$ and random reward $R_j$; now however the controller is unaware of the distribution of $R_j$, and must learn it from observations.
In period $t$, the controller observes the arrival-type $j$, and decides to accept/reject based on observed rewards up to time $t$. 
We consider two feedback models: \emph{full feedback} where the controller observes $R_j$ regardless of whether the item is accepted or rejected, and \emph{censored feedback} where the controller only observes rewards of accepted items; for the latter (which is sometimes referred to as online contextual bandits with knapsacks), we assume the rewards $R_j$ have sub-Gaussian tails~\citep[Section 2.3]{concentration_book}.

\paragraph{Benchmarks and guarantees.} 
Our framework, \rabbi (\emph{Re-solve and Act Based on the Bellman Inequalities}; see \cref{sec:resolve}) is based on comparing two `controllers': \off, who acts optimally given future information, and a non-anticipative controller \onl who tries to follow \off. 
Both start in the same initial state $S^\T$.
We denote $\voff$ as the expected total reward collected by \off acting optimally (i.e., according to a Bellman equation) given its information structure.
In contrast, \onl uses a non-anticipative policy $\pi$ that maps current states to actions, resulting in a total expected reward $\von_{\pi}$. 

Let $\pi_R$ denote the online policy produced by our \rabbi framework, and $\pi$ denote any non-anticipative policy.
Then the expected regret of $\pi_R$ relative to the chosen offline benchmark is
\begin{align*}
\E[\reg] \defeq  \voff-\von_{\pi_R} \geq \max_{\pi}\left[\von_\pi\right]-\von_{\pi_R}
\end{align*}
The last inequality, which follows from the fact that $\von_{\pi}\leq \voff$ for any pair of benchmark and online policies, emphasizes that the regret is a bound on the \emph{additive gap w.r.t. the best online policy}.

For all the above problems, we use the \rabbi framework to identify an appropriate benchmark, with respect to which we get the following guarantees: First, for the \oknap, we recover a result proved in~\cite{multi_secretary,bayes_prophet}
\begin{theorem}[Theorem~1 in~\citet{multi_secretary}]
\label{theo:baseline}
For known reward distributions with finite mean, an online policy based on the \rabbi framework obtains regret that depends only on the primitives $(n,\vp,\vr,\vw)$, but is independent of the horizon length $T$ and resource budget $B$.
\end{theorem}
The above builds intuition for using \rabbi in more complex settings. 
In particular, the benchmark used in~\cref{theo:baseline} is the full-information prophet, which is too loose for obtaining constant regret in the remaining settings (pricing, probing, and bandits; see~\cref{ex:full_info}).
This is where our framework helps in guiding the choice of the right benchmark. 
In particular, we obtain the following results:
\begin{theorem}[Online Probing]\label{theo:reg_probing}
For reward distributions with finite support of size $m$, an online-probing policy based on the \rabbi framework (\cref{alg:probing}) obtains regret that depends only on $(n,m,\vq,\vp,\vr)$, but is independent of horizon length $T$, resource budget $B$ and probing budget $\Bp$. 
\end{theorem}
\begin{theorem}[Dynamic Pricing]\label{theo:reg_pricing}
For any reward distributions $(F_j:j\in [n])$ and prices $\vf$, a pricing policy based on the \rabbi framework (\Cref{alg:pricing}) obtains regret that depends only on $(A,\vf,F_1,\ldots,F_n)$, but is independent of horizon length $T$ and initial budget levels $B\in\N^d$. 
\end{theorem} 
The result for dynamic pricing also extends naturally to resource bundles and general customer-choice models (see~\cref{ssec:consumerchoice} and \cref{theo:reg_multi_pricing} therein).

For the bandit settings, we define a separation parameter $\delta = \min_{j\neq j'} \abs*{\E[R_j]/w_j-\E[R_{j'}]/w_{j'}}$; this is only for our bounds, and is not known to the algorithm.
\begin{theorem}[Knapsack with Distribution Learning] \label{theo:reg_learning} 
Assuming the reward distributions are sub-Gaussian, in the full feedback setting, a policy based on the \rabbi framework (\Cref{alg:bandits}) obtains regret that depends only on the primitives $(n,\vp,\vr,\vw,\delta)$ and is independent of the horizon length $T$ and knapsack capacity $B$. 
\end{theorem}
The last result can also be used as a black-box for the censored feedback setting to get an $O(\log T)$ regret guarantee (see \cref{cor:learning} in \cref{sec:bandits}).

\subsection{Overview of our Framework}\label{sec:framework}

We develop our framework in the full generality of MDPs in~\cref{sec:bellman}.
To give an overview and gain insight into the general version, we use \oknap as a warm-up. A schema for the framework is provided in~\cref{fig:diagram_simple}.

In the \oknap problem, at any time-to-go $t$, let $Z_j^t\in\N$ denote the (random) number of type-$j$ arrivals in the remaining $t$ periods.
Recall rewards of type-$j$ arrivals have expected value $r_j\defeq \E[R_j]$.
Define \off to be a controller that knows $Z^t$ for all $t$ in advance.
The total reward collected by \off can be written as an integer linear program
\begin{equation}\label{eq:knapsack_val}
V(t,b|Z^t) = \max_{\vx_\accept\in\N^n}\crl{\vr'\vx: \vw'\vx_{\accept}\leq b, \vx_{\accept}\leq  Z^t}
=\max_{\vx_\accept,\vx_\reject\in\N^n}\crl{\vr'\vx_\accept: \vw'\vx_\accept\leq b, \vx_{\accept}+\vx_{\reject} = Z^t}.
\end{equation}
The function $V(\cdot|Z^t)$ is thus \off's value function (see \cref{fig:diagram_simple}), where the notation $|Z^t$ emphasizes that $V$ is conditioned on $Z^t$.
Moreover, for every $j$, the variables $x_{\accept,j},x_{\reject,j}$ represent \emph{action summaries}: the number of type-$j$ arrivals accepted and rejected, respectively.

$V(\cdot|Z^t)$ can also be represented via Bellman equations. Specifically, at time-to-go $t$, assuming \off has budget $b$ and the arriving type is $\xi$, the value function obeys the Bellman equation
\begin{equation*}
V(t,b|Z^t) = \max\crl*{\brk*{r_{\xit}+V(t-1,b-w_{\xit}|Z^{t-1})}\In{w_{\xit}\leq b}, V(t-1,b|Z^{t-1})}, \quad \forall t,b,\xit.
\end{equation*}

\begin{figure}[!ht]
	\centering
	\includegraphics[scale=0.8]{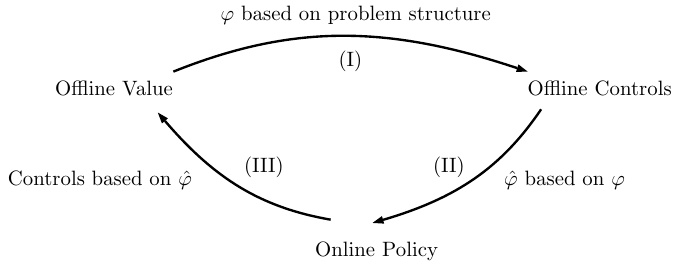}
	\caption{The \rabbi framework: We first define \off's value function by specifying access to future information. 
	Next, we identify a tractable relaxation $\varphi$ for \off's value under this same information structure (step I). 
	Finally, we introduce a non-anticipative estimate $\hatphi$ for $\varphi$, and use it to design online controls (step II). The resulting online policy is evaluated against \off's value (step III).}
	\label{fig:diagram_simple}
\end{figure}

Next consider the linear programming relaxation for $V(t,b)$
\begin{align*}
\varphi(t,b|Z^t) \defeq \max_{\vx_\accept,\vx_\reject\geq 0}\crl{\vr'\vx_\accept: \vw'\vx_\accept\leq b, \vx_{\accept}+\vx_{\reject} = Z^t},
\end{align*}

It is clear that $\varphi$ is more tractable compared to $V$, and also, that it approximates $V$ up to an integrality gap.
However, $\varphi$ \emph{does not obey a Bellman equation}.
To circumvent this, we introduce the notion of \emph{Bellman Inequalities}, wherein we require that $\varphi$ satisfies Bellman-like conditions for `most' sample paths. 
Formally, for some random variables $\BL$, we want $\varphi$ to satisfy 
\begin{equation*} 
\varphi(t,b|Z^t) \leq \max\crl*{\brk{r_{\xit}+\varphi(t-1,b-w_{\xit}|Z^{t-1})}\mathds{1}_{\{w_{\xit}\leq b\}}, \varphi(t-1,b|Z^{t-1})} + \BL(t,b).
\end{equation*} 
Note that, if $\E[\BL(t,b)]$ is small, with expectation taken over $Z^t$, then $\varphi$ `almost' satisfies the Bellman equations.
We henceforth refer to $\varphi$ as a \emph{relaxed value} for $V$ and $\BL$ the \emph{Bellman Loss}.

Establishing that actions derived from $\varphi$ are nearly optimal for \off accomplishes step (I) in \cref{fig:diagram_simple}. For step (II), we want to emulate \off by estimating $\varphi$ based on current information. 
A natural estimate is obtained by taking expectations over future randomness, to get:
\begin{align*}
\hatphi(t,b) \defeq \max_{\vy_\accept,\vy_\reject\geq 0}\crl{\vr'\vy_\accept: \vw'\vy_\accept\leq b, \vy_{\accept}+\vy_{\reject} = \E[Z^t]}.
\end{align*}
Note that $\hatphi$ \emph{does not approximate $V$ or $\varphi$} up to a constant additive error~\cite{bayes_prophet}; however, $\hatphi$ can be used as a predictor for the action taken by \off.
Specifically, at time $t$ with current budget $b$, \rabbi first computes $\hatphi(t,b)$ and then interprets the solution $\vy$ as a score for each action (here, accept/reject).
We show that taking the action with the highest score (i.e., action $\argmax_{u\in\crl{\accept,\reject}}\crl{y_{\xit,u}}$) guarantees that that \onl and \off play the same action with high probability.
Whenever \off and \onl play different actions, then we incur a loss, which we refer to as the {\em Information Loss}, as it quantifies how having less information impacts \onl's actions. 
This process of using $\hatphi$ to derive actions is represented as step (III) in~\cref{fig:diagram_simple}.

\paragraph{Towards a general framework.}
For all the problems in \cref{sec:setting}, our approach uses a similar three-step process, wherein we choose an \off benchmark, identify relaxed value $\varphi$ via appropriate optimization problem, and get an online policy based on estimate $\hatphi$. Consequently, we refer to our framework as \rabbi, which stands for \emph{Re-solve and Act Based on Bellman Inequalities}. 

Our work builds on constant-regret policies for multidimensional packing~\cite{bayes_prophet}, and more general online optimization problems~\citep{banerjee2020uniform}. The techniques developed in these works, however, have two fundamental shortcomings that prevent them from addressing the settings we consider:
\begin{itemize}
\item Use of full-information benchmarks: Existing works~\citep{multi_secretary,bayes_prophet,banerjee2020uniform} use the full information benchmark, which is too loose for our settings. Indeed, for probing/pricing/learning settings, \emph{no algorithm can have constant regret compared to the full information benchmark} (see~\cref{ex:full_info}).
\item Explicit value-function characterizations: The optimization problem in~\cref{eq:knapsack_val} has a closed-form solution, which was used explicitly by~\citep{multi_secretary,bayes_prophet,banerjee2020uniform}; this does not extend to more complex settings.
\end{itemize}
Our framework in this work resolves these shortcomings in a structured way, allowing us to get provably near-optimal algorithms for several canonical resource allocation problems. 
Moreover, we do so via a generalized notion of information-augmented benchmarks, and our decomposition of the regret into the Information Loss (capturing randomness in inputs) and Bellman Loss (capturing limited computational power). This flexibility helps greatly in the design of our algorithms.

\subsection{Related Work}

Our approach has commonalities with two closely related approaches:
\paragraph{Prophet Inequalities and Ex-Ante Relaxations:}
A well-studied framework for obtaining performance guarantees for heuristics policies is to compare against a full information agent, or ``prophet''.
This line of work focuses on competitive-ratio bounds, see \citep{kleinberg2012matroid,duetting2017prophet,correa2017posted} for overviews of the area.
In particular, \citep{correa2017posted} obtains a multiplicative guarantee for dynamic posted pricing with a single item under worst case distribution. 
A related line of work considers the use of ex-ante LP relaxations~\cite{alaei2014bayesian,buchbinder2014secretary} for obtaining worst-case competitive guarantees in online packing problems.
In contrast, we obtain an additive guarantee for multiple items in a parametric setting.

\paragraph{MDP Dual Relaxations.} A standard way to get bounds on MDPs is via information-relaxations, which at a high level, create benchmarks by endowing \off with additional information, while forcing it to `pay a penalty' for using this information.
\citep{info_relaxation,info_relaxation2} use this in a \emph{dual-fitting} approach, to construct performance bounds for greedy algorithms in different problems. 
In contrast, our framework is similar to a \emph{primal-dual} approach: we adaptively construct our relaxations, and derive controls directly from them.
We compare the two approaches in more detail in Appendix~\ref{sec:inforelax}. 

\medskip
Moreover, the different problems we apply \rabbi each have a large body of prior work. 

\paragraph{Online Packing.}
There is a long line of work on the baseline \oknap and generalizations.
A notable work in this line is~\citet{jasin2012}, who gives a policy with constant expected regret when the problem instance is far from a set of certain \emph{non-degenerate} instances.
This inefficiency, though, is fundamental, since they use the ex ante (or fluid) benchmark, which has $\Omega(\sqrt{T})$ under non-degeneracy. More recently, \citep{wang_resolve} partially extend the result of~\citep{multi_secretary} for more general packing problems; however their policy only gives constant regret under i.i.d. Poisson arrivals, and require the system to be scaled linearly (i.e., $B$ grows proportional to $T$).
In contrast,~\citep{multi_secretary} (one dimension) and~\citep{bayes_prophet} (multiple dimensions) provide constant regret policies with no assumption on the scaling. 
The approach in the latter is further generalized in~\cite{banerjee2020uniform} to handle more complex problems including bin-packing and QOS constraints.
See~\cite{bayes_prophet,banerjee2020uniform} for more discussion and references.

\paragraph{Probing.} 
Approximation algorithms have been developed for {\em offline} probing problems, both under budget constraints~\citep{probing_applications} and probing costs~\citep{weitzman1979optimal,price_info}. 
Another line of work pursues tractable {\em non-adaptive} constant-factor competitive algorithms for this problem~\citep{adaptivity_gaps}. In terms of {\em online adaptive} algorithms,~\citet{probing_submodular} introduces an algorithm with bounded competitive ratio in an adversarial setting.

\paragraph{Dynamic posted pricing.} This is a canonical problem in operations management, with a vast literature; see \citet{talluri_book} for an overview. 
Much of this literature focuses on asymptotically optimal policies in regimes where the inventory $B$ and/or horizon $T$ grow large. 
When $B$ and $T$ are scaled together by a factor $k$, there are known algorithms with regret that scales as $O(\sqrt{k})$ or $O(\log(k))$, depending on assumptions on the primitives (e.g., smoothness of the demand with price)~\citep{jasin_pricing}.
There is also vast literature on pricing when the demand function is not known and has to be learned~\citep{jasin_learning}. 
Finally, under adversarial arrivals,~\citet{dynamic_pricing_limited} provides a policy with $O((B\log T)^{2/3})$ regret under \emph{adversarial inputs}, as opposed to our $O(1)$ guarantee under stochastic inputs.

\paragraph{Knapsack with learning.} Multi-armed bandit problems have been widely studied, and we refer to \citet{bubeck,bubeck_book} for an overview.
Bandit problems with combinatorial constraints on the arms are known as Bandits With Knapsacks~\citep{bandits_with_knapsacks}, and the generalization where arms arrive online is known as Contextual Bandits With Knapsacks~\citep{resourceful_bandits,linear_contextual}. Results in this literature typically  study worst-case distributions. We, in contrast, pursue parametric regret bounds that explicitly depend on the (unknown) discrete distribution. 
Closest to our work is~\citet{srikant}, who provide a UCB-based algorithm that gets $O(\sqrt{T})$ regret (in contrast, we get $O(\log T)$ regret for the same setting).

\section{Approximate Control Policies via the Bellman Inequalities}
\label{sec:bellman}

In this section, we describe our general framework.
Before proceeding, we introduce some notation: We work an underlying probability space $(\Omega,\Sigma,\Pr)$, and for any event $\calB\subseteq\Omega$, we denote its complement by $\calB^c$.
We use boldface letters to indicate vector-valued variables (e.g. $\vp, \vw$, etc.), and capital letters to denote matrices and/or random variables. 
For an optimization problem $(P)$, we use $P$ to denote its optimal value.
When using LP formulations with decision variables $\vx$, we interchangeably use $x_{ij}=x(i,j)$ to denote the $(i,j)^{th}$ component of $\vx$.

\subsection{Offline Benchmarks and Bellman Inequalities}
\label{subsec:offline}

We consider an online decision-making problem with state space $\S$ and action space $\U$, evolving over periods $t=\T,\T-1,\ldots,1$; here $T$ denotes the horizon, and $t$ is the {time to-go}. In any period $t$, the controller first observes a random arrival $\xit\in\Xi$, following which it must choose an action $\u\in\U$. 
For system-state $s\in\S$ at the beginning of period $t$, and random arrival $\xi\in\Xi$, an action $\u\in\U$ results in a reward $\Re(s,\xi,\u)$, and transition to the next state $\Tr(s,\xi,\u)$. 
We assume both reward and future state are random variables whose realizations are determined for every $\u$ given $\xi$. 
This assumption is for ease of exposition only; our results can be extended to hold when rewards or transitions are random given $\xi$.

The feasible actions for state $s$ and input $\xi$ correspond to the set $\crl{\u\in\U:\Re(s,\xi,\u)>-\infty}$.
We assume that this feasible set is non-empty for all $s\in\S,\xi\in\Xi$, and also, that the maximum reward is bounded, i.e., $\sup_{s\in \S,\xi\in\Xi,\u\in\mathcal{U}}\Re(s,\xi,\u)<\infty$.

The MDP described above induces a natural filtration $\calF$, with $\calF_t = \sigma(\crl{\xi^\tau:\tau \geq t})$; a non-anticipative policy is one which is adapted to $\calF_t$.
We allow \off to use a \emph{richer} information filtration $\calG$, where $\calG_t\supseteq \calF_t$. Note that since $t$ denotes the time-to-go, we have $\calG_{t-1}\supseteq \calG_t$. 
Henceforth, to keep track of the information structure, we use the notation $f(\cdot|\calG_t)$ to clarify that a function $f$ is measurable with respect to the sigma-field $\calG_t$.

Given any filtration $\calG$, \off is assumed to play the optimal policy adapted to $\calG$, hence \off's value function is given by the following Bellman equation:
\begin{equation}
\label{eq:bellman_mcdp}
	V(t,s|\calG_t) = \max_{\u\in\U}\crl{\Re(s,\xit,\u) +\E[V(t-1,\Tr(s,\xit,\u)|\calG_{t-1})|\calG_t]},
\end{equation}
with the boundary condition $V(0,\cdot)=0$. 
We denote the expected value as $\voff\defeq\E[V(T,S^T|\calG_T)]$. 
Note that $\voff$ is an upper bound on the performance of the optimal non-anticipative policy. 

We present a specific class of filtration (generated by augmenting the  canonical filtration) that suffice for our applications (see \cref{fig:canonical} for an illustration of the definition).
\begin{definition}[Canonical augmented filtration]
	\label{def:canonical_filtration}
	Let $G_\Theta \defeq (G_\theta: \theta\in\Theta)$ be a set of random variables.
	The canonical filtration w.r.t.\ $G_\Theta$ is 
	\begin{align*}
		\calG_t = \sigma(\crl{\xi^l:l\geq t}\cup G_\Theta)
		\supseteq \calF_t.
	\end{align*}
\end{definition}
The richest augmented filtration is the \emph{full information} filtration, wherein $\calG_t=\calF_1$ for all $t$, i.e., the canonical filtration with $G_\Theta = (\xit:t\in [T])$.
As $\calG_t$ gets coarser, the difference in performance between \off and \onl decreases.
Indeed, when $\calG=\calF$, then \cref{eq:bellman_mcdp} reduces to the Bellman equation for the value-function of an optimal non-anticipative policy:
\begin{equation*}
	V(t,s|\calF_t) = \max_{\u\in\U}\crl{\Re(s,\xit,\u) +\E[V(t-1,\Tr(s,\xit,\u)|\calF_{t-1})]},
	\quad V(0,\cdot,\cdot) = 0,
\end{equation*} 
where the expectation is taken with respect to the next period's input $\xi^{t-1}$.

\begin{example}[Full Information is Too Loose]\label{ex:full_info}
Consider a dynamic pricing instance with $n=d=1$, prices $\vf = (1,2)$, and valuation distribution $\Pr[R^t=1+\varepsilon] = p$ and $\Pr[R^t=2 + \varepsilon] = 1-p$.
When $B=T$, the optimal policy always posts a price that maximizes $\left(\f\cdot \Pr[R^t > \f]\right)$.
If $p\geq 1/2$, then the optimal policy (DP) always posts price $\f=1$ and has expected reward $T$.
On the other hand, full information can post price $R^t -\varepsilon$ at time $t$ and extract full surplus $\voff = \sum_t \E[R^t -\varepsilon] = T(2-p)$.
Thus the regret against full information must grow as $\Omega(T)$.
\emph{This example is not pathological}; the same behavior persists even in random instances (see \cref{sec:pricing_numeric}).
\end{example}

\begin{figure}[t]
    \centering
    \includegraphics[scale=0.65]{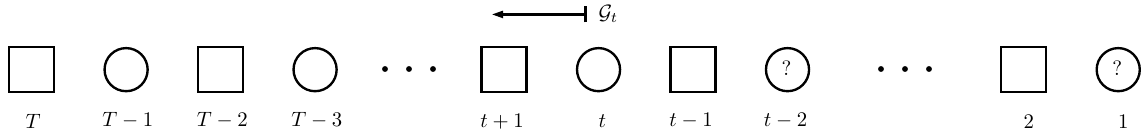}
    \caption{Illustration of \cref{def:canonical_filtration}.
    In online probing (see~\cref{sec:probing}), arrivals first reveal their public type, then the controller chooses an action (accept/probe/reject), and then the private type (true reward) is revealed.
    Squares (resp. circles) represent public (resp. private) information.
    The filtration $\calG$ used by \rabbi comprises of all public types, i.e.\ $G_\Theta =(\xi^\theta: \xi^\theta \text{ is a public type})$.
    At time $t$, \off knows all the information thus far (to the left and including $t$), plus the future squares.
    }
    \label{fig:canonical}
\end{figure}

We are now ready to introduce the notion of relaxed value $\varphi$ and Bellman Inequalities. 
Intuitively, $\varphi$ is ``almost'' defined by a dynamic-programming recursion; quantitatively, whenever $\varphi$ does not satisfy the Bellman equation, we incur an additional loss $\BL$, which we denote the Bellman loss.

\begin{definition}[Bellman Inequalities]
\label{def:relaxed_bellman}
The family of r.v.\ $\crl{\varphi(t,s)}_{t,s}$  satisfies the Bellman Inequalities w.r.t.\ filtration $\calG$ and r.v.\ $\crl{\BL(t,s)}_{t,s}$ if $\varphi(t,\cdot)$ and $\BL(t,\cdot)$ are $\calG_t$-measurable for all $t$ and the following conditions hold:
\begin{enumerate}
\item Initial ordering: $\E[V(\T,S^T)|\calG_T] \leq \varphi(\T,S^\T|\calG_{\T})$.
\item Monotonicity: $\forall s\in\S,t\in[T]$, 
\begin{equation}\label{eq:relaxed_bellman}
\varphi(t,s|\calG_t) \leq \max_{\u\in\U}\crl{\Re(s,\xit,\u)+\E[\varphi(t-1,\Tr(s,\xit,\u)|\calG_{t-1})|\calG_t]} + \BL(t,s).
\end{equation}
\item Terminal Condition: $\varphi(0,s)=0\,\forall\,s\in\S$
\end{enumerate}
We refer to $\varphi$ and $\BL$ as the \emph{relaxed value} and  \emph{Bellman loss} pair with respect to $\mathcal{G}$, and use $|\calG_t$ to remind the reader that we need the information contained in $\calG_t$ to evaluate $\varphi(t,s)$
\end{definition} 
Given any $\varphi$, monotonicity holds trivially with $\BL=\varphi$ (but leads to poor performance guarantees).
On the other hand, $\varphi$ (which may be intractable) is the only value function guaranteeing $\BL=0$.
The crux of our approach is to identify a good $\varphi$ balances the loss and tractability.

A special case is when the Bellman Loss is $0$ over sample paths in some chosen set:
\begin{definition}[Exclusion Sets]
\label{def:exclusion}
A set $\calB(t,s)$ is an \emph{exclusion set} if we can write the Bellman Loss as $\BL(t,s)=\rmax \Ins{\calB(t,s)}$ for some constant $\rmax>0$ and events $\calB(t,s)\subseteq\Omega$.
\end{definition}
If the Bellman Loss can be defined with exclusion sets, then from \cref{def:relaxed_bellman} (monotonicity) we obtain the condition $\varphi(t,s|\calG_t) \leq \max_{\u\in\U}\crl{\Re(s,\xit,\u)+\E[\varphi(t-1,\Tr(s,\xit,\u)|\calG_{t-1})|\calG_t]}$, i.e., monotonicity is satisfied for all realizations $\omega\in \Omega$ except for those in the exclusion set $\calB(t,s)$.

To build intuition, we specify the Bellman Inequalities for our baseline \oknap.
For this end, we first need the following lemma characterizing the sensitivity of LP solutions.

\begin{lemma}\label{lem:collect} 
Consider an LP $(P[\vd]):\max\{\vr'\vx:M\vx=\vd,\vx\geq 0\}$ , where $M\in\R^{m\times n}$ is an arbitrary constraint matrix. 
If $\bar{\vx}$ solves $(P[\vd])$ and $\bar{x}_j\geq 1$ for some $j$, then $P[\vd]=r_j+P[\vd-M_j]$. 
\end{lemma} 
\begin{proof} By assumption, the optimal value of $(P[\vd])$ remains unchanged if we add the inequality $x_j\geq 1$. 
Therefore we have $P[\vd]=\max\crl{\vr'(\vx+\ve_j):M(\vx+\ve_j)=\vd,\vx\geq 0}$.  
\end{proof}

\cref{lem:collect} lets us divide $P[\vd]$ in two summands: the immediate reward $r_j$ and the future reward $P[\vd-M_j]$; this has the flavor of dynamic programming we need for defining the Bellman loss.

\begin{example}[Bellman Loss For Baseline Setting]\label{ex:indicator}
For the baseline \oknap, discussed in \cref{sec:framework}, we chose the full information filtration $\calG_t=\calF_1$ for all $t$ so that $\varphi(t,b|\calG_t) \defeq \max_{\vx\geq 0}\crl{\vr'\vx_\accept: \vw'\vx_\accept\leq b, \vx_{\accept}+\vx_{\reject} = Z^t}$. 
We define the exclusion sets as 
\begin{align*}\calB(t,b) = \crl{\omega\in\Omega: \not\exists \vx \text{ solving } \varphi(t,b) \text{ s.t. } x(\accept,\xit)\geq 1 \text { or } x(\reject,\xit)\geq 1}.\end{align*}
By \cref{lem:collect}, outside the exclusion sets $\calB(t,b)$, monotonicity holds with zero Bellman Loss, i.e.,
\begin{align*}
\varphi(t,s|\calG_t) \leq \max_{\u\in\U}\crl{\Re(s,\xit,\u)+\E[\varphi(t-1,\Tr(s,\xit,\u)|\calG_{t-1})|\calG_t]} \quad\forall \omega\notin \calB(t,s).
\end{align*}
Moreover, for our choice of $\varphi$, since the optimal solution sorts items by $r_j/w_j$, we have that the maximum loss outside the exclusion set is bounded by $\rmax\leq \max_{j,i}\crl{w_ir_j/w_j-r_i}$, which depends only on the primitives.
Thus,~\cref{def:relaxed_bellman} is satisfied with Bellman Loss $\BL(t,b)=\rmax\Ins{\calB(t,b)}$ .
\end{example}

To generalize this, we need two definitions. First, we define the maximum Bellman loss as:
\begin{definition}[Maximum Loss] \label{def:loss}
	For a given relaxation $\varphi$, the maximum loss is given by
	\begin{align*}
	\rmax \defeq \max_{t,s,\u:\Re(s,\xit,\u)>-\infty}\crl{ \varphi(t,s|\calG_t) - (\Re(s,\xit,\u)+\E[\varphi(t-1,\Tr(s,\xit,\u)|\calG_{t-1})|\calG_t])}
	\end{align*}
\end{definition}
Next, note that the `optimal' action in the RHS of \cref{eq:relaxed_bellman} need not be unique, and indeed the inequality can be satisfied by multiple actions.
For given $\varphi$ and $\BL$, we define:
\begin{definition}[Satisfying actions]\label{def:satisfying}
	Given  a filtration $\calG$ and relaxed value $\varphi$, we say that $\u$ is {\em a satisfying action} for state $s$ at time $t$ if 
	\begin{equation}\label{eq:satisfying}
		\varphi(t,s|\calG_t) \leq \Re(s,\xit,\u)+\E[\varphi(t-1,\Tr(s,\xit,\u)|\calG_{t-1})|\calG_t] +\BL(t,s).
	\end{equation}
\end{definition}
At any time $t$ and state $s\in\S$, any action in $\argmax_{\u\in\U}\crl{\Re(s,\xit,\u)+\E[\varphi(t-1,\Tr(s,\xit,\u)|\calG_{t-1})|\calG_t]}$ is always a satisfying action (see monotonicity in \cref{def:relaxed_bellman}); moreover, to identify a satisfying action, we must know $\calG_t$. We now have the following proposition.
\begin{proposition}\label{prop:relaxed_bellman}
	Consider a relaxation $\varphi$ and Bellman loss $\BL$ that satisfy the Bellman inequalities w.r.t.\ filtration $\calG$.
	Let $(\St,t\in [T])$ denote the state trajectory under a policy that, at time $t$, takes \emph{any} satisfying action $\Ut=\Ut(\St|\Gt)$.
	Then,
	\begin{align*}
	\E\brk*{V(\T,S^\T|\mathcal{G}_T)} - \E\brk*{ \sum_{t=1}^\T \Re(\St,\xit,\Ut)}\leq \E\brk*{ \sum_{t=1}^\T\BL(t,\St|\Gt)}. 
	\end{align*}
\end{proposition}
\begin{proof}
	From the monotonicity condition in the Bellman inequalities (\cref{def:relaxed_bellman}), and the definition of a satisfying action (\cref{def:satisfying}), we have, for all time $t$, that 
	\begin{equation*}
		\varphi(t,\St|\calG_t)\leq \E[\Re(\St,\xit,\Ut)+\varphi(t-1,S^{t-1}|\calG_{t-1}) +  \BL(t,\St|\calG_t)|\calG_t].
	\end{equation*}
	Iterating the above inequality over $t$ we get $\varphi(T,S^T|\calG_T)\leq  \sum_{t=1}^\T \E[\Re(\St,\xit,\Ut)+\BL(t,\St|\calG_t)|\calG_t]$.
	Finally, by the initial ordering condition we have $\E[V(\T,S^\T)|\calG_T] \leq \varphi(\T,S^\T|\calG_\T)$.
\end{proof}

\cref{prop:relaxed_bellman} shows that a policy that always plays a satisfying action $\Ut$ \emph{approximates} the performance of \off up to an additive gap given by the \emph{total Bellman loss} $\E\brk*{\sum_{t=1}^T \BL(t,\St|\calG_t)}$. 
More importantly, it suggests that \onl should try to track \off by `guessing' and playing a satisfying action $\Ut$ in each period.
We next illustrate how \onl can generate such guesses. 

\subsection{From Relaxations to Online Policies} 
\label{sec:resolve}

Suppose we are given an augmented canonical filtration $\calG_t = \sigma(\crl{\xi^l:l\geq t}\cup G_\Theta)$, and assume that the relaxed value $\varphi$ can be represented as a function of the random variables $\crl{\xi^l:l\geq t} \cup G_\Theta$ as $ \varphi(t,s|\calG_t) = \varphi(t,s;f_t(\xi^T,\ldots,\xit,G_\Theta))$.
In particular, we henceforth focus on a special case where $\varphi$ is expressed as the solution of an optimization problem: 
\begin{align}
	\label{eq:generic_relaxation}
	\varphi(t,s;f_t(\xi^T,\ldots,\xit,G_\Theta))
	= \max_{\vx\in\R^{\U\times\Xi}} \crl{h_t(\vx;s,f_t(\xi^T,\ldots,\xit,G_\Theta )) : g_t(\vx;s,f_t(\xi^T,\ldots,\xit,G_\Theta))\leq 0}.
\end{align}
The decision variables give \emph{action summaries}: for given state $s$ and time $t$, $x_{\u,\xi}$ represents the number of times action $\u$ is taken for input $\xi$ in remaining periods.
We can also interpret $x_{\u,\xi}$ as a \emph{score} for action $u$ when input $\xi$ is presented.
Now to get a non-anticipative policy, a natural `projection' of $\varphi(t,s|\calG_t)$ on the filtration $\calF$ is given via the following optimization problem 
\begin{align}
	\label{eq:hatvarphi}
	\hat{\varphi}(t,s|\calF_t) & = \varphi(t,s;\E[f_t(\xi^T,\ldots,\xit,G_\Theta)|\calF_t]) 
	= \max_{\vy\in\R^{\U\times\Xi}} \crl{h_t(\vy;s,\E[f_t|\calF_t]) : g_t(\vy;s,\E[f_t|\calF_t])\leq 0}.
\end{align}
The solution of this optimization problem gives action summaries (or scores) $\vy$; the main idea of the RABBI algorithm is to play the action with the highest score.

\begin{algorithm}
	\floatname{algorithm}{RABBI}
	\renewcommand{\thealgorithm}{}
	\caption{(Re-solve and Act Based on Bellman Inequalities)}
	\label{alg:resolve}
	\begin{algorithmic}[1]
		\Require Access to functions $f_t$  such that $\varphi(t,s|\calG_t) = \varphi(t,s;f_t(\xi^T,\ldots,\xit,G_\Theta))$.
		\Ensure Sequence of decisions $\hUt$ for \onl.
		\State Set $S^\T$ as the given initial state
		\For{$t=\T,\ldots,1$}
		\State Compute $\hatphi(t,S^t)=\varphi(t,S^t;\E[f_t(\xi^T,\ldots,\xit,G_\Theta)|\calF_t])$ with associated scores $\vy=\crl{y_{\u,\xi}}_{u\in\U,\xi\in\Xi}$
		\State Given input $\xit$, choose the action $\hUt$ with the highest score $y_{\u,\xit}$
		\State Collect reward $\Re(\St,\xit,\hUt)$; update state $S^{t-1} \gets \Tr(\St,\xit,\hUt)$
		\EndFor
	\end{algorithmic}
\end{algorithm}

\begin{theorem}
	\label{theo:resolve}
	Let \off be defined by an augmented filtration $\calG_t$ as in \cref{def:canonical_filtration}.
	Assume the relaxation $\varphi(t,s)$ satisfies the Bellman Inequalities with loss $\BL$, and for all $(t,a)\in[T]\times\S$, let $\calQ(t,s)\subseteq \Omega$ denote the set of sample-paths where the action $\hUt$ taken by \rabbi is not a satisfying action.
	If $(\St,t\in[T])$ denotes the state trajectory under \rabbi, then
	\begin{align*}
	\E[\reg] \leq \E\brk*{\sum_t (\rmax\Ins{\calQ(t,\St)} +\Ins{\calQ(t,\St)^c}\BL(t,\St))}
	\leq \sum_t \left(\rmax\Pr[\calQ(t,\St)] + \E[\BL(t,\St)]\right).
	\end{align*}
\end{theorem}
\begin{remark}[Bellman and Information Loss]
	The bound in \cref{theo:resolve} has two distinct summands:
	The \emph{information loss} $\sum_t \Pr[\calQ(t,\St)]$ measures how often \rabbi takes a non-satisfying action due to randomness in sample paths;
	on the other hand, the \emph{Bellman loss} $\sum_t \E[\BL(t,\St)])$ quantifies violations of the Bellman equations made under the pseudo value-function $\varphi$.
\end{remark}


\paragraph{Compensated Coupling:} 
The proof of~\cref{theo:resolve} is based on the \emph{compensated coupling} approach introduced in~\citep{bayes_prophet}. 
The idea is to imagine`simulating' controllers \off and \onl with identical random inputs $(\xit:t\in[T])$, with \onl acting before \off. 
Moreover, suppose at some time $t$, both controllers are in the same state $s$. 
Recall that, for any given state $s$ at time $t$, an action $\u$ is satisfying if \off's value does not decrease when playing $\u$ (\cref{def:satisfying}).
If \onl chooses to play a satisfying action, then we can make \off play the same action, and consequently both move to the same state.
On the other hand, if \onl chooses an action that is not satisfying, then the two trajectories may separate; we can avoid this however by `compensating' \off so that it `agrees' to take the same action as \onl. 
In particular, its always sufficient to compensate \off by the \emph{maximum loss} $\rmax$ to ensure its reward does not decrease by following \onl.
As a consequence, \emph{the (compensated) \off and \onl take the same actions, and thus their trajectories are coupled}.

As an example, for \oknap with budget $B=2$, weights $w_j= 1\,\forall\,j$, and horizon $\T=5$, consider a sample-path $\omega\in\Omega$ with rewards  $(\xi^5,\xi^4,\xi^3,\xi^2,\xi^1)=(5,7,2,7,2)$.
The sample-path comprises of three different types, and the sequence of actions $(\reject,\accept,\reject,\accept,\reject)$ ((selecting the value $7$ items) is optimal for \off, with total reward of $14$. 
Suppose \onl, in period $t=5$ wants to accept the item with reward $\xi^5=5$; then, \off is ``willing'' to follow this action if given a compensation of $2$ (in addition to collected reward $5$). \off and \onl then start the next period $t=4$ in the same state with budget $1$, hence remain coupled.

\begin{proofof}{\cref{theo:resolve}}
	Denoting \off's state as $\bar\St$, we have via \cref{prop:relaxed_bellman} that $\forall\,t$:
	\begin{align*}
	\varphi(t,\bar\St|\calG_t)\leq \E[\Re(\bar\St,\xit,\Ut)+\varphi(t-1,\bar S^{t-1}|\calG_{t-1}) +  \BL(t,\bar\St)|\calG_t].
	\end{align*}
	
	Let us assume as the induction hypothesis that $\bar\St=\St$.
	This holds for $t=T$ by definition.
	At any time $t$ and state $S^t$, if $\hUt$ is not a satisfying action for \off, then we have from the definition of the maximum loss (\cref{def:loss}) that:
	\begin{align*}\rmax \geq 
	\varphi(t,\St|\calG_t)-\Re(\St,\xit,\hUt)+\E[\varphi(t-1,S^{t-1}|\calG_{t-1})|\calG_t])  \quad a.s..
	\end{align*}
	Now to make \off take action $\hUt$ so as to have the same subsequent state as $\onl$, it is sufficient to compensate \off with an additional reward of $\rmax$.
	Specifically, we have
	\begin{align*}
	\varphi(t,\St|\calG_t)\leq \E[\Re(\St,\xit,\hUt)+\varphi(t-1,S^{t-1}|\calG_{t-1}) +  \rmax\Ins{\calQ(t,\St)} +\Ins{\calQ(t,\St)^c}\BL(t,\St)|\calG_t].
	\end{align*}
	Finally, as in \cref{prop:relaxed_bellman}, we can iterate over $t$ to obtain
	\begin{align*}
	\E[\varphi(T,S^T|\calG_T)] \leq \E\brk*{ \sum_t\Re(\St,\xit,\hUt) +   \sum_t (\rmax\Ins{\calQ(t,\St)} +\Ins{\calQ(t,\St)^c}\BL(t,\St))}.
	\end{align*}
	The first sum on the right-hand side corresponds exactly to \onl's total reward using the \rabbi policy.
	By the initial ordering property, $\E[V(T,S^T)]\leq \E[\varphi(T,S^T)]$, and we get the result.
\end{proofof}

\section{Online Probing}
\label{sec:probing}

We now apply our framework to online probing. Here, each arrival type $j$ has an independent random reward $R_j\in\crl{r_{jk}: k\in[m]}$ drawn with probabilities $\{q_{jk}\}$; $\vr$ and $\vq$ are known.
We assume w.l.o.g that $r_{j1} <r_{j2}<\ldots<r_{jm}$ and $r_{jm}>0$. 
For ease of exposition, we assume that all arrivals have unit weights; our analysis however extends to general weights $w_j$. 
The controller may accept ($\accept$), reject ($\reject$) or probe ($\probe$) the arrival. Accepting type-$j$ item without probing results in expected reward of $\bar r_j\defeq \sum_{k\in [m]}r_{jk}q_{jk}$. 
Probing reveals the realized reward, after which it can be accepted or rejected.
The controller has a resource budget $\Bh\in \N$ and a probing budget $\Bp\in\N$. 
When an arrival is accepted (resp. probed), we reduce $\Bh$ (resp. $\Bp$) by one.

Formally, we view each time period $t\in\{T,T-1,\ldots,1\}$ as comprising of a mini dynamic program with two stages  $\crl{t,t-1/2}$, driven by external random inputs $\xit\in [n]$ and $\xi^{t-1/2}\in [n]\times[m]$.
In the first stage $t$, the controller observes the arriving request $\xit=j$, and chooses an action in $\crl{\accept,\probe,\reject}$; in the second stage $t-1/2$, the reward $r_{jk}$ (or ``sub-type" $\xi^{t-1/2}=(j,k)\in [n]\times [m]$) is drawn with probability $q_{jk}$, and the available actions are $\crl{\accept,\reject}$ if the first-stage action is $\probe$, and $\varnothing$ otherwise. 
We augment the state space with a variable $\diamond$ that captures the first stage decision (i.e., whether we accept/reject without probing or probe). 
The state space $\S$ of the controlled process is thus 
$
\S= \crl{(b_h,b_p,\diamond):b_h,b_p\in\N,\diamond\in \crl{\accept,\probe,\reject,\varnothing}},
$
where $b_h,b_p$ are the residual hiring and probing budgets. 
In first stage of each period, we set  $\diamond=\varnothing$, and only collect rewards in second stage in each period.
See~\cref{fig:probing_states} for an illustration.

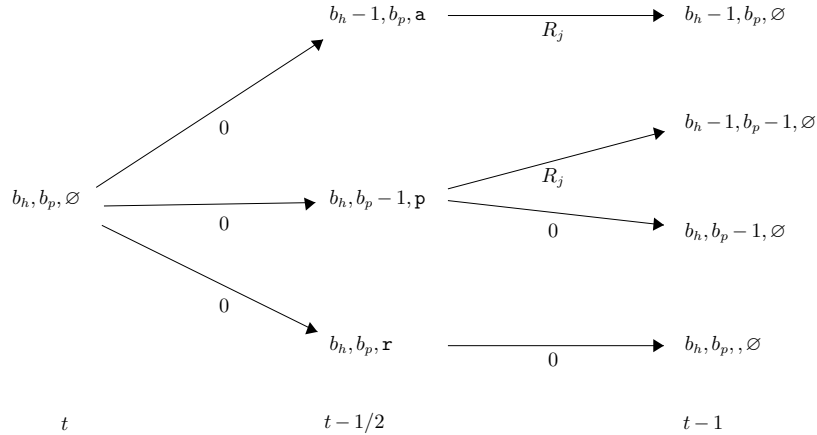
\begin{figure}[!ht]
	\centering
	\scalebox{0.7}{%
		\input{probing_states}
	}
	\caption{Actions/transitions in online probing in periods $t$, $t-1/2$, and $t-1$, with inputs $\xit=j$ and $\xi^{t-1/2}=R^j$. Numbers below the arrows represent the reward of a transition. 
		At $t$, available actions are $\crl{\accept,\probe,\reject}$ (i.e., accept, probe, reject; from top to bottom);
		at $t-1/2$, if we chose to probe in the first-stage (i.e., are in the middle state), then available actions are $\crl{\accept,\reject}$.
	}
	\label{fig:probing_states}
\end{figure}

\subsection{Offline Benchmark and Online Policy for Probing}
\label{ssec:probingbenchmark}

We now apply the \rabbi framework for online probing.

\paragraph{\off Benchmark:} 
We define \off to be the controller that knows the public types of {\em all} arrivals in advance (i.e., it knows $Z_j^{t}$, the number of type-$j$ items that will arrive in the last $t$ periods), but does not know the realization of the rewards (sub-types).
Formally, \off is endowed with the canonical filtration given by $\Theta = [T]$ and $G_\theta=\xi^\theta$ (see \cref{def:canonical_filtration}): with $t$ steps to go, \off has the information filtration $\calG_t=\sigma(\crl{\xi^{t}:t\in[\T]}\cup\crl{\xi^\tau:\tau\geq t})$.
Note that since \off does not know the actual rewards, it still needs to solve a dynamic program to decide whether or not to probe an arrival.

\paragraph{Relaxed Value Function:} 
Since solving for \off's optimal actions may be non-trivial, we next construct a relaxed value function $\varphi$, using the following LP parametrized by $(b_h,b_p,\vz)\in \N^2\times\Rp^n$, 
\begin{alignat}{2}
\label{eq:probing_lp}
(P[b_h,b_p,\vz]) \qquad & \text{maximize: } & & \sum_{j,k}r_{jk}x_{jk\accept}+ \sum_j \bar{r}_jx_{j\accept} \\
& \text{subject to: }& \quad & 
\begin{aligned}[t]
\sum_{j,k}x_{jk\accept}+\sum_jx_{j\accept} & \leq b_h \nonumber\\
\sum_{j}x_{j\probe} & \leq b_p \nonumber\\
x_{j\accept} +x_{j\probe} + x_{j\reject} & = z_j& \forall\,j & \in [n] \nonumber\\
x_{jk\accept} +x_{jk\reject} & = q_{jk}x_{j\probe} & \forall\,j & \in [n], k \in [m] \nonumber\\
\vx & \geq 0 \nonumber
\end{aligned}
\end{alignat}
Intuitively, $P[b_h,b_p,\vz]$ can be understood as follows: given current resource and probing budgets $\vb$ and future arrivals $\vz$, the decision variables $\vx\in\Rp^{3n+2nm}$ represent action summaries, where $x_{j\accept},x_{j\reject},x_{j\probe}$ are the total number of future type-$j$ arrivals that are accepted without probing, rejected without probing, and probed respectively, and $x_{jk\accept},x_{jk\reject}$ are the number of probed future type-$j$ arrivals that are revealed to have reward $r_{jk}$, and then accepted/rejected respectively. 
The first two constraints implement the resource budget and probing budget; the third ensures the number of type-$j$ items accepted, probed or rejected equals arrivals of that type. 
Finally, the last constraint guarantees that a $q_{jk}$ fraction of probed type-$j$ items have sub-type $k$ (i.e., reward $r_{jk}$).

To construct relaxed value $\varphi$, recall that a state is of the form $s=(b_h,b_p,\diamond)$ with $\diamond\in \crl{\accept,\probe,\reject,\varnothing}$. 
For period $t$ (i.e., first stage, $\diamond=\varnothing$), we define $\varphi(t,(b_h,b_p,\varnothing)|\calG_t) \defeq P[b_h,b_p,Z^t]$.
For $t-1/2$ (i.e., second stage decisions), we modify $\varphi$ to incorporate the action $(\accept,\probe,\reject)$ taken in the first stage.
Overall, our relaxation is defined as follows
\begin{equation}\label{eq:probing_relax}
\varphi(t-1/2,(b_h,b_p,\diamond)|\calG_t) = \left\{ 
\begin{array}{ll}
r_{\xi^{t-1/2}}+P[(b_h,b_p),Z^{t-1}] & \quad\diamond = \accept \\ 
\max\crl{r_{\xi^{t-1/2}}+P[(b_h-1,b_p),Z^{t-1}],P[(b_h,b_p),Z^{t-1}]} & \quad\diamond = \probe\\
P[(b_h,b_p),Z^{t-1}] &  \quad\diamond = \reject 
\end{array} 
\right.
\end{equation}

\paragraph{Value Function estimate and Online Policy:}
Finally we can use the relaxed value function $\phi$ in~\cref{eq:probing_relax} to construct an estimated value function $\hatphi$ by replacing $Z^t$ with $\E_{\xi^{t-1/2}}[Z^t]$. Using this, we get our online policy specified in~\cref{alg:probing}.

\begin{algorithm}
\caption{Probing \rabbi}
\label{alg:probing}
\begin{algorithmic}[1]
\Require Access to solutions of $(P[\vb,\vz])$
\Ensure Sequence of decisions for \onl.
\State Initialize budgets $(\Bh^{T},\Bp^{T})\gets (\Bh,\Bp)$
\For{period $t=T,\ldots,1$}	
    \State Compute $\Xt$, an optimal solution to $(P[\Bt,\E[Z^{t}]])$
	\State Observe the arrival, say it is of type $j$, then take action $\hUt\in\argmax_{u=\accept,\probe,\reject}\crl{\Xt_{j,u}}$.
	\State If $\hUt=\reject$ or $\hUt=\accept$: collect zero or random $R_{j}$, respectively.
	\State If $\hUt=\probe$: probe the arrival to observe $R_j=r_{jk}$, then take action $\argmax_{u=\accept,\reject}\crl{\Xt_{j,k, u}}$
	\State Update budgets $B^{t-1}$ accordingly.
\EndFor
\end{algorithmic}
\end{algorithm}
\begin{remark}[Probing cost] Our approach can also handle a setting where the controller has no probing budget, but instead incurs a penalty $c_j$ when probing a type-$j$ arrival. 
The only change to results and proofs is in the definition of $P[\vb,Z]$, where we drop the constraint involving the probing budget, and modify the objective to be $\max\crl{\sum_{j,k}r_{jk}x_{jk\accept}+ \sum_j \bar{r}_jx_{j\accept} -\sum_jc_jx_{j\probe}}$
\end{remark}

\subsection{Regret Analysis for Online Probing}
\label{ssec:probeoutline}

We now provide a brief outline of the proof of~\cref{theo:reg_probing}, which guarantees that~\cref{alg:probing} has a regret that is independent of $T,\Bh$ and $\Bp$.
Complete proofs are provided in~\cref{sec:probingproofs}.

The main part of the proof involves showing that $\varphi$ as defined in~\cref{eq:probing_relax} obeys the Bellman inequalities (\cref{def:relaxed_bellman}) with appropriately chosen Bellman loss. The first ingredient for this is provided by the following lemma, which establishes initial ordering for our relaxed value $\varphi$.
\begin{lemma}
\label{lem:probing_initial}
For any $b_h,b_p\in\N$, and arrivals $Z$, $\E[V(T,(b_h,b_p)|\calG_{T})] \leq \E[\varphi(\T,(b_h,b_p,\varnothing)|\calG_{T})]$. 
\end{lemma}
This follows from a standard argument, where we argue that any offline policy induces action summaries that satisfy the constraints defining $\varphi$. 
The proof is provided in~\cref{sec:probingproofs}. 

The bulk of the work is in establishing monotonicity, which we do via the following lemma. Recall the definitions of exclusion sets, satisfying actions and maximum loss (\cref{def:exclusion,def:satisfying,def:loss}). 
\begin{lemma}
\label{lem:probing_ineq}
Let $\bar X$ be a maximizer of $(P[(b_h,b_p),Z^{t}])$ for some period $t$, and suppose $\xit=i$. Then we have the following implications for satisfying actions
\begin{itemize}
\item[\em (1)] If $\bar X_{i\accept}\geq 1$, then accepting at time $t$ is a satisfying action.
\item[\em (2)] If $\bar X_{i\reject}\geq 1$, then rejecting at time $t$ is a satisfying action.
\item[\em (3)] If $\bar X_{i\probe}\geq 1$, and $\xi^{t-1/2}=(i,k)$ is such that either $\bar X_{ik\accept}\geq 1$ or $\bar X_{ik\reject}\geq 1$, then probing at time $t$, followed by accepting (if $\bar X_{ik\accept}\geq 1$) or rejecting (if $\bar X_{ik\reject}\geq 1$) at time $t-1/2$ is a satisfying action.
\end{itemize}

Finally $\varphi$ satisfies the Bellman Inequalities with Bellman Loss $\BL(t,(b_h,b_p))=\rmax\Ins{\calB(t,b_h,b_p)}$, where $\calB$ are exclusion sets defined as: 
\begin{align*}
\calB(t,b_h,b_p) = \crl{\omega\in\Omega: \not\exists \bar X \text{ solution to } (P[(b_h,b_p),Z^{t}]) \text{ s.t. } (1) \textbf{ or } (2) \textbf{ or } (3) \text{ hold} }.
\end{align*}
\end{lemma}
The proof generalizes the argument in~\cref{ex:indicator} for \oknap. 
We provide a brief outline here, and defer the details to~\cref{sec:probingproofs}. 
First, observe that the monotonicity condition in~\cref{def:relaxed_bellman} translates to the following condition in the online probing setting.
\begin{align*}
\varphi(t,(b_h,b_p,\varnothing)|\calG_t) \leq \max_{\diamond\in\crl{\accept,\probe,\reject}}\crl{\E_{\xi^{t-1/2}}[\varphi(t-1/2,(s_{\diamond},\diamond)|\calG_{t-1/2})|\calG_t]} \quad \forall \omega\notin\calB(t,b_h,b_p).
\end{align*}
where the state $s_{\diamond}=(b_h-1,b_p)$ if $\diamond=\accept$, $s_{\diamond}=(b_h,b_p-1)$ if $\diamond=\probe$ and $s_{\diamond}=(b_h,b_p)$ if $\diamond=\reject$. 
Moreover, given $\xit=i$, we have from~\cref{eq:probing_relax} that $\E_{\xi^{t-1/2}}[\varphi(t-1/2,(s_{\diamond},\diamond))|\calG_{t-1/2})|\calG_t]=P[(b_h,b_p),Z^{t-1}] $ if $\diamond = \reject$, and $r_{\xi^{t-1/2}}+P[(b_h-1,b_p),Z^{t-1}] $ if $\diamond = \accept$. 
Now for cases (1) and (2), the claim in the lemma follows directly by invoking~\cref{lem:collect}.
Finally, case (3) (where $\bar X_{i\probe}\geq 1$) also follows from using~\cref{lem:collect}, but in a somewhat more technical way; see~\cref{sec:probingproofs} for details. 

Using~\cref{lem:probing_initial,lem:probing_ineq}, we can complete the regret analysis for~\cref{alg:probing}.

\begin{proofof}{\cref{theo:reg_probing}}
By \cref{theo:resolve}, we have that $\reg \leq \rmax\sum_t (\Ins{\calB(t,\St)} + \Ins{\calQ(t,\St)})$.
To bound this, we proceed in two steps: bounding the measure of the exclusion sets $\calB$, and the ``disagreement'' sets $\calQ$. We conclude using the fact that $\rmax\leq\max_{j,k}r_{jk}$.
	
To bound the measure of the exclusion sets $\calB$, let $\bar X$ be the solution to $(P[\vb,Z^{t}])$, and note that \cref{lem:probing_ineq} guarantees that there is zero Bellman Loss if (1)
$\max\crl{\bar X_{j\accept},\bar X_{j\reject}}\geq 1$, or (2) $\bar X_{j\probe}\geq 1$ and $\max\crl{\bar X_{jk\accept},\bar X_{jk\reject}}\geq 1$.
The exclusion set $\calB(t,\vb)$ comprises sample paths where both (1) and (2) fail.
	
Note that any feasible solution to $(P[\vb,Z^{t}])$ satisfies $x_{j\accept}+x_{j\probe}+x_{j\reject}=Z_j^{t}$ $\forall j$ and $x_{jk\accept}+x_{jk\reject}=q_{jk}x_{j\probe}$ $\forall j,k$.
If $Z_j^{t}\geq 3$, then one of the variables $x_{j\accept},x_{j\probe},x_{j\reject}$ must be at least $1$.
On the other hand, we need $q_{jk}x_{j\probe}\geq 2$ to guarantee that one of $x_{jk\accept},x_{jk\reject}$ is at least $1$.
Thus we have 
\begin{equation} 
\label{eq:interim}
\Pr[\calB(t,b)|\xi^{t-1/2}=(j,k)] \leq \Pr\brk*{Z_j^{t} < \frac{6}{q_{jk}}}
=\Pr\brk*{Z_j^{t} -\mu_j(t) < -\mu_j(t)\prn*{1-\frac{6}{\mu_j(t)q_{jk}}}}.
\end{equation} 
Restricting $\mu_j(t)\geq 12/q_{jk}$ to ensure the RHS of \cref{eq:interim} is positive, we can use a standard Chernoff bound (see \citep{concentration_book}) to get
$\Pr[\calB(t,b)|\xi^{t-1/2}=(j,k)]  \leq e^{-2(p_j/2)t}+\In{t\leq 12/(p_jq_{jk})}
$. 
Finally, 
\begin{align*}
\sum_{t}\Pr[\calB(t,B^t)]\leq \sum_t\sum_jp_je^{-2(p_j/2)t} + \sum_t\sum_{j,k}p_jq_{jk}\In{t\leq 12/(p_jq_{jk})}\leq \sum_j\frac{2}{p_j}+12.
\end{align*}

To bound the Information Loss $\sum_t\Pr[\calQ(t,\St)]$, recall $\calQ(t,\St)\subseteq\Omega$ is the event where $\hUt$ is not satisfying.
Let $\bar X$ be a solution to $(P[\vb,Z^{t}])$, $t$ a first stage, and let $j=\xit$.
We now have two cases depending on if $\hUt\in\crl{\accept,\reject}$ or $\hUt=\probe$.
First, if $\hUt\in\crl{\accept,\reject}$, then according to \cref{lem:probing_ineq}, accepting or rejecting is satisfying whenever $\max\{\bar X_{j\accept},\bar X_{j\reject}\}\geq 1$.
Since $\Xt(\xit,\hUt)=\max\crl{\Xt(\xit,u):u=\accept,\probe,\reject}$ and $\Xt_{j\accept}+\Xt_{j\probe}+\Xt_{j\reject}=\mu_j(t)$, we have
\begin{align*}
\Pr[\bar X(j,\hUt)<1|\Xt(j,\hUt)\geq \mu_j(t)/3]
	\leq \Pr[\norm{\bar X-\Xt}_\infty \geq \mu_j(t)/3 ].
\end{align*}
On the other hand, if $\hUt=\probe$, the error is bounded by
\begin{align*}
\Pr\brk*{\bar X_{j\probe}<1 \text{ or } \bar X_{\xi^{t-1/2},u}<1 \Big|\Xt_{j\probe}\geq \frac{\mu_j(t)}{3}, \Xt_{\xi^{t-1/2},u}\geq \frac{q_{\xi^{t-1/2}}\mu_j(t)}{6} }
\leq \Pr\brk*{\norm{\bar X-\Xt}_\infty \geq \frac{q_{\xi^{t-1/2}}\mu_j(t)}{6}},
\end{align*}
where $u$ is the action with largest value between the variables $\Xt(\xi^{t-1/2},\accept),\Xt(\xi^{t-1/2},\reject)$.
	
Thus, regardless of the action $\hUt$, the probability of choosing a non-satisfying action is bounded by $\Pr[\norm{\bar X-\Xt}_\infty \geq  \min_kq_{jk}\cdot \mu_j(t)/6 ]$. 
Moreover, standard LP sensitivity results~\citep[Theorem 2.4]{mangasarian} imply that there exists $\kappa$ depending on $\vq,n,m$ alone, s.t.\ $\norm{\bar X-\Xt}_\infty\leq\kappa \norm{Z^{t}-\mu(t)}_1$. 
Finally, the measure of sets $\calQ$ where \onl chooses a non-satisfying action is bounded by
\begin{equation*}
\sum_t\Pr[\calQ(t,\St)] \leq \sum_t \Pr[\norm{Z^{t}-\mu(t)}_1\geq \min_kq_{jk}\cdot \mu_j(t)/6\kappa ]<\infty.
\end{equation*}
The summability follows arguments presented in~\citep{bayes_prophet}, based on standard concentration bounds.
\end{proofof}

\section{Dynamic Pricing}
\label{sec:pricing}

We now apply our framework to dynamic pricing. In the basic setting, we have $d$ resources and $n$ customer types.
Each customer type has a private reward for a set of resources.
The controller observes the customer type, and if the corresponding set of resources is available, posts a price. 
The customer then purchases iff the requested set is available and the posted price below the private reward.
The resource consumption is encoded in a matrix $A\in\crl{0,1}^{d\times n}$.
In~\cref{ssec:consumerchoice}, we generalize to settings where rather than requesting a specific set of products, customers make a choice between multiple substitute bundles of resources.

We consider the following formal model: at time $t$, type $j\in[n]$ arrives with probability $p_j$, is seen by the controller, who then posts a price $\f_{jl}$ from a set of available prices $\crl{\f_{j1},\ldots,\f_{jm}}$.
The customer then draws a private reward $R^t\sim F_j$, and a purchase occurs iff $R^t>\f_{jl}$.
If the customer buys, $\f_{jl}$ is collected and the inventory decreases by $A_j$.
On the other hand, if the customer does not buy, the controller collects zero and the inventory remains unchanged.

\subsection{Offline Benchmark and Online Policy for Dynamic Pricing}

\paragraph{\off Benchmark:} 
Note that for each customer type $j$, there are $Z_j^{T}$ arrivals, and hence $Z_j^{T}$ draws from the distribution $F_j$.
We now define our benchmark by considering \off to be a controller that knows the realized histogram of these draws, i.e., for each $j$, \off knows the empirical distribution of the $Z_j^{T}$ rewards.
Moreover, at the end of each period $t$, \off also observes the realized valuation $R^t$ whether or not there is a sale. 
Note that \off does not know the exact sequence of these rewards, and so is not a full information benchmark.
For example, say $Z_1^T=15$ and we reveal that $10$ arrivals type-1 have private reward $\$1$ and $5$ arrivals type-1 have private reward $\$ 2$; now, upon observing a type-1 arrival, \off concludes that the reward is $\$ 2$ with probability $\frac{5}{15}$.
Now if the arrival had value $\$1$, then, the next time \off observes a type-1, its belief is that the reward is $\$2$ with probability $\frac{5}{14}$.

Formally for each $j$ suppose the prices are ordered $\f_{j1}>\f_{j2}>\ldots>\f_{jm}$.
Denote $\xit\in[n]$ to be the type of the arrival at time $t$.
To define \off, we introduce a sequence of independent random vectors $\crl{Y^t:t=T,T-1,\ldots,1}$ where $Y^t_{jl}\defeq\In{\xit=j,R^t> \f_{jl}}$; in other words, $Y^t_{jl}$ is the indicator of  whether a price $f_{jl}$ or lower is accepted by the type-$j$ at time $t$.
We define $Q_{jl}(t)\defeq \frac{1}{Z_j^{t}}\sum_{\tau=1}^tY_{jl}^\tau$ to be the fraction of type-$j$ customers who accept price $\f_{jl}$ in the last $t$ periods.
Observe that $Q_{jl}(t)$ is a martingale with $\E[Q_{jl}(t)] = \bar F_j(\f_{jl})$ and $Q_{jl}(t) = \frac{Z_j^{t+1}}{Z_j^{t}} Q_{jl}(t+1) - \frac{1}{Z_j^{t}}Y^{t+1}_{jl}$.

\off's information is now given by the filtration $\calG_t=\sigma(\crl{Q(\tau),Z^\tau:\tau\geq t})$, i.e., at every time $t$, \off knows the total demand $Z_j^{t}$ and the empirical averages $Q_{jl}(t)$, but not the sequence of rewards.
This coincides with the canonical filtration (\cref{def:canonical_filtration}) with variables $(Q_{jl}(T),Z_j^{T}:j\in[n],l\in [m])$.
The filtration $\calG$ is strictly coarser than the full information filtration, which would correspond to revealing all the variables $Y^T,Y^{T-1},\ldots,Y^1$ instead of their empirical averages.

\paragraph{Relaxed Value Function:} 
Consider the following LP, parameterized by $(\vb,\vq,\vz)$.
\begin{alignat}{2}
\label{eq:pricing_lp}
(P[\vb,\vq,\vz])\qquad & \text{maximize: } & & \sum_{j,l}\f_{jl}q_{jl}x_{jl} \\
& \text{subject to: }& \quad & 
\begin{aligned}[t]
\sum_{j,l}a_{ij}q_{jl}x_{jl} & \leq b_i& \forall\,i & \in[d] \nonumber\\
\sum_{j,l} x_{jl} + x_{jr} & = z_j& \forall\,j & \in [n] \nonumber\\
\vx & \geq 0 \nonumber
\end{aligned}
\end{alignat}
We define the relaxed value as $\varphi(t,\vb|\calG_t)\defeq P[\vb,Q(t),Z^{t}]$, and the corresponding estimated value as $\hatphi(t,\vb)\defeq P[\vb,\vq,t\vp]$, where $q_{jl} = \bar F_j(\f_{jl})$. 
The resulting \rabbi policy is presented in \cref{alg:pricing}.

\begin{algorithm}
\caption{Pricing \rabbi}
\label{alg:pricing}
\begin{algorithmic}[1]
\Require Access to solutions of $(P[\vb,\vq,\vz])$
\Ensure Sequence of decisions for \onl.
\State Set $B^{\T}\gets B$ as the given initial budget and $q_{jl}\gets \bar F_j(\f_{jl})$
\For{$t=\T,...,1$}
    \State If the arrival is type $j$ and $A_j \not \leq B^t$: not enough resources, reject and go to $t-1$
	\State Compute $\Xt$, an optimal solution to $(P[B^t,\vq,t\vp])$
	\State Let $l\in\argmax\crl{\Xt_{j,l}:l=1,\ldots,m,\reject}$. If $l=\reject$, reject and go to $t-1$. Else post price $\f_{jl}$ 
	\State If $R^t > \f_{jl}$, collect $\f_{jl}$ and $B^{t-1}\gets B^t-A_j$; else $B^{t-1}\gets B^t$
\EndFor
\end{algorithmic}
\end{algorithm}

To get some intuition into the LP $(P[\vb,\vq,\vz])$, note that if $q_{jl} = \bar F_j(\f_{jl})$, i.e., the probability that price $\f_{jl}$ is accepted by a type-$j$ customer and $z_j$ is the number of type-$j$ arrivals, then $(P[\vb,\vq,\vz])$ can be interpreted as follows: the variable $x_{jl}$ represents the number of times that price $\f_{jl}$ is offered, with $\sum_{j,l}\f_{jl}q_{jl}x_{jl}$ the expected reward from the corresponding arrivals.
Each time price $\f_{jl}$ is offered, $a_{ij}q_{jl}$ units of resource $i$ are consumed in expectation, and hence $\sum_{j,l}a_{ij}q_{jl}x_{jl}$ is the total expected consumption of resource $i$.
Finally, at most one price is offered per arrival, which is captured by $\sum_{l}x_{jl}+x_{j\reject} =z_j$, where $x_{j\reject}$ is the number of rejected type-$j$ customers.

\subsection{Bellman Inequalities and Bellman Loss}

We first argue that our choice of $\varphi$ satisfies the Bellman Inequalities.
\begin{lemma}
\label{lem:pricing_ineqs}
Let $V(T,B|\calG_T)$ be the value of \off's optimal policy, and $\varphi(t,\vb|\calG_t)= P[\vb,Q(t),Z^{t}]$ be the relaxed value with optimal solution $X$.
\begin{enumerate}
    \item $\E[V(T,B|\calG_T)] \leq \E[\varphi(T,B|\calG_T)]$, hence $\varphi$ satisfies the initial ordering condition.
    \item If the arriving type is $j$ and $\max_l\crl{X_{jl}}\geq 1$, then $\E[\BL(t,\vb)]\leq 0 $.
	\item If the arriving type is $j$ and $X_{jl}\geq 1$, then posting $\f_{jl}$ is a satisfying action.
\end{enumerate}
\end{lemma}
We omit the proof of the initial ordering in item (1), as it is similar to that of \cref{lem:probing_initial}. Below we present the main ingredients for obtaining the monotonicity property (items (2) and (3)); complete details are deferred to~\cref{sec:pricingproofs}.
For ease of exposition, when the controller rejects, he can equivalently post $f_{j\reject}=\infty$ such that $\bar F_j(f_{j\reject})=0$ with the convention $0\times\infty=0$.

We start by recalling the monotonicity condition (\cref{def:relaxed_bellman}).
Denote $\E_t[\cdot]=\E[\cdot |\calG_t]$.
If the inventory is $\vb\geq A_j$, the random reward of posting price $\f_{jl}$ at $t$ is $\f_{jl}Y^t_{jl}$ and the random new inventory is $\vb-A_jY^t_{jl}$, thus monotonicity corresponds to:
\begin{align*}
\varphi(t+1,\vb) \leq \max_{l\in[m]\cup\crl{\reject}}\crl{\E_{t+1}[\f_{jl}Y^{t+1}_{jl}+\varphi(t,\vb-A_jY^{t+1}_{jl})]} + \E_{t+1}[\BL(t+1,\vb)].    
\end{align*}
Because $Q$ is a martingale, we have $\E_{t}[Y^t] = Q(t)$ and we can further simplify the condition to
\begin{equation}\label{eq:mon_pricing}
\varphi(t+1,\vb) \leq \max_{l\in[m]\cup\crl{\reject}}\crl{ \f_{jl}Q_{jl}(t+1) + \E_{t+1}[\varphi(t,\vb-A_jY^{t+1}_{jl})] }+ \E_{t+1}[\BL(t+1,b)].
\end{equation}

Define
$
\BL(t+1,b,j,l) \defeq \varphi(t+1,\vb) - \f_{jl}Q_{jl}(t+1) - \E_{t+1}[\varphi(t,\vb-A_jY^{t+1}_{jl})], 
$
which corresponds to the loss in \cref{eq:mon_pricing} when we assume a specific price $\f_{jl}$ is posted. Recall we define $\varphi(t+1,\vb)=P[\vb,Q(t+1),Z^{t+1}]$. 
Moreover, for an arrival of type $j$ and any solution $X$ of $P[\vb,Q(t+1),Z^{t+1}]$, if $X_{jl}\geq 1$, then using \cref{lem:collect}, we have $P[\vb,Q(t+1),Z^{t+1}] = \f_{jl}Q_{jl}(t+1)+P[\vb-A_jQ_{jl}(t+1),Q(t+1),Z^{t}]$. 
Thus, assuming $X_{jl}\geq 1$, we can write the loss in the Bellman inequality as
\begin{equation}
    \BL(t+1,\vb,j,l) = P[\vb-A_jQ_{jl}(t+1),Q(t+1),Z^{t}] - \E_{t+1}[ P[\vb-A_jY^{t+1}_{jl},Q(t),Z^{t}]]
    \label{eq:losspricing}
\end{equation}

Observe that $\BL(t,\vb,j,l)$ is characterized by a random LP that depends on $Y^{t+1}$ (which is unknown at time $t+1$), see \cref{eq:losspricing}. 
To complete item (2) of~\cref{lem:pricing_ineqs}, it remains to prove that $\BL(t,\vb,j,l)$ characterized in \eqref{eq:losspricing} satisfies $\E_t[\BL(t,\vb,j,l)]\leq 0$. 
This is proved in \cref{sec:pricingproofs} by arguing that the term in \eqref{eq:losspricing} is upper bounded by a zero-mean random variable. 

We can then conclude that, for each $l$ with $X_{jl}\geq 1$, $\E_{t+1}[\BL(t+1,b,j,l)]\leq 0$ so that
$ \varphi(t+1,\vb) \leq \E_{t+1}[\f_{jl}Q_{jl}(t+1) + \varphi(t,\vb-A_jY^{t+1}_{jl})], 
$ implying that posting price $\f_{jl}$ is a satisfying action, which is item (3) of~\cref{lem:pricing_ineqs}.

\subsection{Information Loss and Overall Performance Guarantee}

Next we study the disagreement sets $\calQ(t,\Bt)$, and bound the information loss $\sum_t\Pr[\calQ(t,\Bt)]$.

\begin{proposition}
\label{lem:information_loss_pricing}
Let $X$ be a solution of $(P[\vb,Q(t),Z^{t}])$.
If $X_{jl}\geq 1$, then posting $\f_{jl}$ is a satisfying action.
Furthermore, the information loss is bounded by $\Pr[\calQ(t,\Bt)]\leq 1/t^2$ for all $t\geq c$, where $c$ depends only on $(\vf,\vp,A,F_1,\ldots,F_n)$.
\end{proposition}

\begin{figure}
    \centering
    \includegraphics[width=0.55\textwidth]{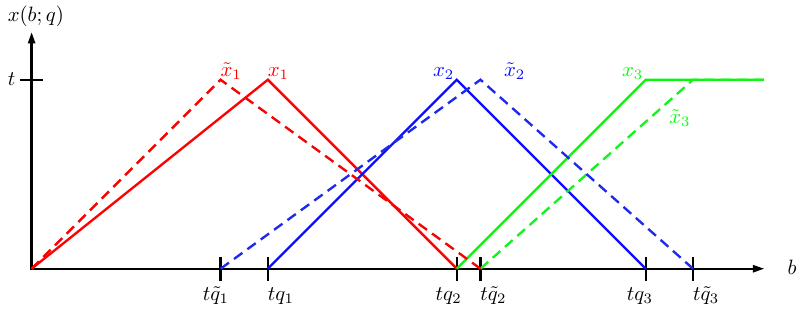}
    \caption{Solution to the pricing LP in \cref{eq:pricing_lp} for the case $d=1$ and $n=1$, which correspond to selling multiple copies of an item to homogeneous customers.
    If $b/t\in (q_l,q_{l+1}]$, the prices used by the LP are $\f_l,\f_{l+1}$ and the amount of time we offer each is piece-wise linear in the budget.
    For a perturbation $\tilde \vq$ of $\vq$, we superpose the solutions with the different parameters.
    Our `guess' is incorrect only when $\tilde x_l\gg1$ and $x_l<1$, which necessitates a substantial perturbation of $\vq$.
    }
    \label{fig:pricing_sol}
\end{figure}

We now give an outline of this proof; for details, refer~\cref{sec:pricingproofs}.
Recall that \rabbi chooses $l$ as the maximum entry of the solution to $(P[\vb,\E[Q(t)],\E[Z^{t}])$, which is a perturbed version of the object of interest, thus \onl needs to guess $l$ such that $X_{jl}\geq 1$ without the knowledge of $Q(t)$ and $Z^{t}$, creating an information loss.

To build intuition, consider the case where $d=1$ and $n=1$, i.e., selling multiple copies of an item to homogeneous customers; since there is only one type, we drop the index $j$.
Recall $\f_1>\ldots>\f_m$ and $q_1<\ldots<q_m$. 
It is easy to check that the solution of $P[b,\vq,t]$ is as follows:
(i) If $b \leq tq_1$, then $x = (b/q_1,0,\ldots,0)$;
(ii) If $b > tq_m$, then $x=(0,\ldots,0,t)$; 
(iii) Otherwise, if $b \in (tq_l,tq_{l+1}]$, then $x_{l'}=0$ for $l'\neq l,l+1$, and $x_l = (tq_{l+1}-b)(q_{l+1}-q_l), x_{l+1} = (b-tq_l)(q_{l+1}-q_l)$.
\Cref{fig:pricing_sol} illustrates this solution, and also shows that for \rabbi's guess to be incorrect, $Q(t)$ and $\E[Q(t)]$ must deviate considerably; the next lemma indicates is unlikely.
This intuition carries over to higher dimensions.

\begin{lemma}\label{lem:dkw}
For any $j\in[n]$, there is a constant $c_j$ depending on $p_j$ only such that, for any time $t$, $\Pr\brk*{\max_l\abs{Q_{jl}(t)-\E[Q_{jl}(t)] }>\sqrt{\frac{\log(t)}{t}}}\leq \frac{c_j}{t^2}$ .
\end{lemma}
\begin{proof}
From the DKW inequality~\citep{dkw} for empirical measures, we have
$$\Pr\brk*{\sup_l\abs{Q_{jl}(t)-\bar F(\f_{jl})}>\lambda\Big| Z^{t}}\leq 2e^{-2\lambda^2Z_j^{t}}.$$ 
Also for $Z_j^{t}\sim \Bin(t,p_j)$, $\E[e^{-\theta Z_j^t}]=(1-p+pe^{-\theta})^t$.
Setting $\lambda =\sqrt{\log(t)/t}$, we get
\begin{align*}
\Pr\brk*{\sup_l\abs{Q_{jl}(t)-\bar F(\f_{jl})}>\sqrt{\frac{\log(t)}{t}}}\leq 2(1-p_j+p_je^{-\theta})^t \quad \text{ where } \theta =2\log(t)/t.
\end{align*} 
Using the inequality $e^{-\theta}\leq 1-\theta+\theta^2/2$, an algebraic check confirms the desired inequality.
\end{proof}

\paragraph{Stability of Left-Hand Side Perturbations.} 
As stated in \cref{alg:pricing}, \onl takes actions based on $P[\vb,\E[Q(t)],\E[Z^{t}]]$, while \off uses $P[\vb,Q(t),Z^{t}]$.
Therefore, for fixed $(t,\vb)$, we need to compare solutions of $P[\vb,\vq,\vz]$ to those of $P[\vb,\vq+\Delta\vq,\vz+\Delta\vz]$, where $\Delta$ is the perturbation.
Define $\vq = \E[Q(t)]$, $\vz=\E[Z^{t}]$, $\Delta\vq = Q(t)-\E[Q(t)]$, and $\Delta\vz=Z^{t}-\E[Z^{t}]$.

\begin{lemma}[Selection Program]\label{lem:sel}
Let $V_t = P[\vb,\vq +\Delta\vq,\vz+\Delta\vz]$ and fix a component $(j',l')$.
Then posting price $\f_{j'l'}$ is satisfying if $\SP[V_t,\vq +\Delta\vq,\vz+\Delta\vz]\geq 1$, where \begin{align*}
\SP[V_t,\vq +\Delta\vq,\vz+\Delta\vz]:=   \max\crl*{x_{j'l'}: \sum_{j,l}f_{jl}(q_{jl}+\Delta q_{jl})x_{jl} \geq V_t, \vx \text{ feasible for } P[\vb,\vq +\Delta\vq,\vz+\Delta\vz] }.
\end{align*}
In other words, $\calQ(t,b,l) = \crl{\omega\in\Omega: \SP[V_t[\omega],Q(t),Z^{t}]< 1}$.
\end{lemma}
\begin{proof}
This problem selects, among all the solutions of $P[\vb,\vq+\Delta\vq,\vz+\Delta\vz]$, one with the largest component $X_{j'l'}$.
From~\cref{lem:pricing_ineqs} we know that, if $X_{j'l'}\geq 1$, then posting $\f_{j'l'}$ is satisfying. 
\end{proof}

We have converted the condition ``$\exists X$ solving $P[v,\vq+\Delta\vq,\vz+\Delta\vz]$ with $X_{j'l'}\geq 1$'' to an optimization program. 
Let $\bar\vx$ be the solution to the proxy $P[\vb,\vq,\vz]$ and let $v_t$ be the objective value (recall that $V_t$ is the value of $P[\vb,\vq+\Delta\vq,\vz+\Delta\vz]$). 
Since the algorithm picks the price with the largest component, assume $\bar x_{j'l'}=\max_{l}\bar x_{j'l}\gg 1$.
In particular, $\SP[v_t,\vq,\vz]\gg 1$ for this fixed $(j',l')$. 
We want to show that $\SP[V_t,\vq +\Delta\vq,\vz+\Delta\vz]\geq 1$ for that particular $(j',l')$.
To that end, we need to bound the difference between $\SP[V_t,\vq +\Delta\vq,\vz+\Delta\vz]$ and $\SP[v_t,\vq,\vz]$. 
This difference depends on (i) $v_t-V_t$, (ii) $\Delta$, and (iii) the dual variables of $(\SP[V_t,\vq +\Delta\vq,\vz+\Delta\vz])$.
Observe that the quantities (i)-(iii) are random. 
We state the result below; the proof is provided in~\cref{sec:pricingproofs}.

\begin{lemma}\label{lem:pricing_dual}
There is a constant $c$ that depends only on $(\vf,\vp,A,F_1,\ldots,F_n)$ such that, for all $t\geq c$, with probability $1-c/t^2$, 
$\SP[V_t,Q(t),Z^{t}]-\SP[v_t,\E[Q(t)],\E[Z^{t}]]\geq -c\sqrt{t\log(t)}$ .
\end{lemma}

\cref{lem:pricing_dual} leads to the bound in \cref{lem:information_loss_pricing}.
Indeed, since the LP in \cref{eq:pricing_lp} has the constraint $\sum_{l\in[m]\cup \crl{\reject}}\bar x_{jl} = tp_j$, the maximum entry is guaranteed to have a value of at least $tp_j/(m+1)$.
Therefore, by definition of the selection program, $\SP[v_t,\E[Q(t)],\E[Z^{t}]]\geq tp_j/(m+1)$.
We know that posting $f_{jl'}$ is satisfying whenever $\SP[V_t,Q(t),Z^{t}]\geq 1$ (see \cref{lem:sel}), hence posting the maximum entry is satisfying provided that $tp_j/(m+1)-c\sqrt{t\log(t)}\geq 1$, which holds for all $t$ large enough.

\subsection{Numerical Simulations}
\label{sec:pricing_numeric}

\begin{figure}[b]
    \centering
    \includegraphics[scale=0.7]{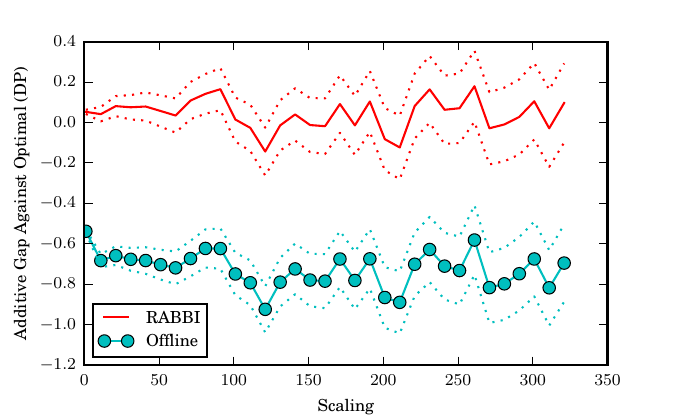}%
    \includegraphics[scale=0.7]{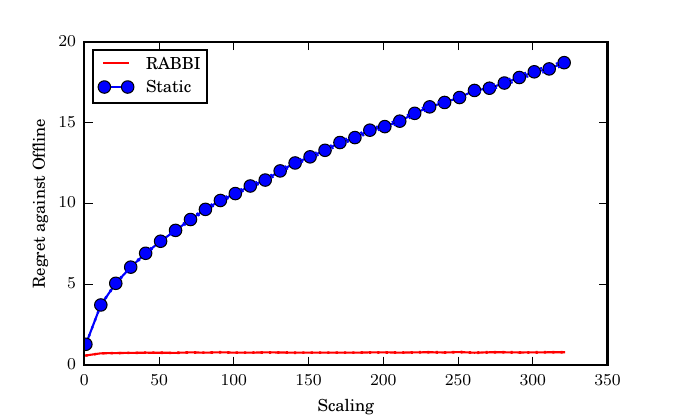}
    \caption{Regret in the `small system' ($n=1$ and $d=1$), with horizon $T=20k$ and initial budget $B=6k$, under scaling $k = 1, 10, 20,\ldots,340$.
    Dotted lines represent 90\% CI.
    (LEFT) additive gaps against the optimal policy, i.e., $V^{DP}-V^{\rabbi}$ and $V^{DP}-V^{\off}$ (RIGHT) Regret of two policies against \off.
    }
    \label{fig:numeric}
\end{figure}

We test our algorithm on two systems, henceforth the ``small system'' and the ``large system''. 
For each system, we consider a sequence of instances with increasing horizons and initial inventories. 

The small system corresponds to the one-dimensional problem ($n=1$ and $d=1$); in this case we can solve the DP for small enough horizons and directly compute the optimality gap. 
The large system corresponds to a multi-dimensional problem with $n=20$, $d=25$ and $m=3$.
The DP solution is intractable for the large system, yet we can compute the offline benchmark and compare our algorithm against it. The optimality gap, recall, is bounded by the offline vs. \rabbi gap.

For the small system, the $k$-th instance has budget $B=6k$ and horizon $T=20k$.
For each scaling $k$, we run $100,000$ simulations.
We consider the following primitives: prices are $(1,2,3)$ and the private reward $R^t$ has an atomic distribution on $(1,2,3)$ with probabilities $(0.3,0.4,0.3)$.
The instance is chosen such that it is dual degenerate for~\eqref{eq:pricing_lp} which is supposedly the more difficult case \cite{jasin_pricing}.
For large system, the parameters were generated randomly and are reported in Appendix~\ref{sec:table}, the $k$-th instance has horizon $T=100k$ and budgets $B_i=10k$ for all $i\in[25]$.

For the small system we consider $k$ small enough (short horizon) so that we can compute the optimal policy; this computation becomes intractable already for moderate values of $k$ (\rabbi however scales gracefully with $k$ as it only requires re-solving an LP in each period).  
In~\cref{fig:numeric} (LEFT) we display the gap between the optimal solution and both the \rabbi and \off's value.
We make two observations: (i) the \off benchmark outperforms the optimal (as it should), but by a rather small margin, and  (ii) \rabbi has a constant regret (i.e., independent of $k$) relative to \off, and hence constant optimality gap. 
In contrast, a full information benchmark would outperform the optimal by too much to be useful. 

\begin{figure}[t]
    \centering
    \includegraphics[scale=0.7]{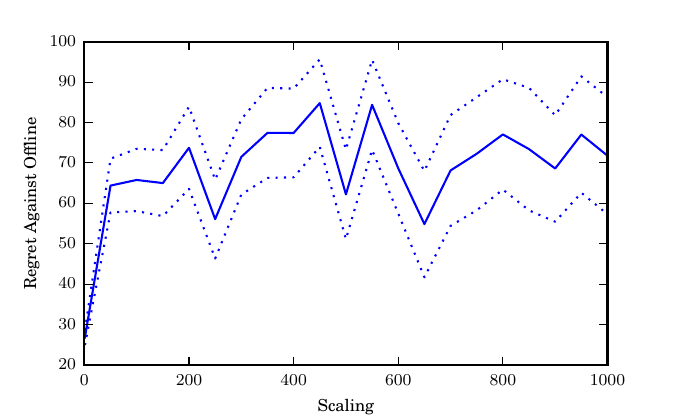}%
    \includegraphics[scale=0.7]{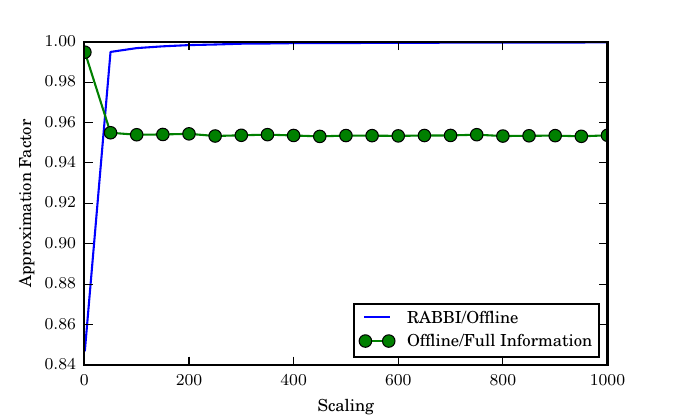}
     \caption{Performance in the `large system' ($n=20$ and $d=25$) with horizon $T=100k$ and initial budgets $B_i=10k$ for $i\in[25]$, under scaling $k = 1,2,\ldots,1000$.
    (LEFT) Regret against \off, i.e., $V^{\off}-V^{\rabbi}$;
    dotted lines represent 90\% CI.
    (RIGHT) Approximation factors $V^{\rabbi}/V^{\off}$ of \rabbi compared to \off, and $V^{\off}/V^{\text{Full-Info}}$ of \off against the full-information benchmark; this shows that the full-information benchmark is indeed too loose, as it is $\Omega(T)$ away from the DP.
    }
    \label{fig:numeric_big}
\end{figure}

In \cref{fig:numeric} (RIGHT), we compare \rabbi to the optimal static pricing policy which has regret $\Omega(\sqrt{k})$ \citep{gallego_pricing}. 
In particular, if $D(\f)$ denotes the demand at fare $f$, we choose the static price to be the one that maximizes the revenue function $\f\cdot D(\f)=\f\cdot T \cdot \bar F(\f)$ subject to the constraint $D(\f)\leq B$.
The solution is the better of two prices: (i) the market clearing price, i.e., that satisfies $D(\f)=B$ or (ii) the monopoly price which maximizes $\f D(\f)$. We note though that when a continuum of prices are allowed,~\citep{jasin_pricing} propose an algorithm (that, like \rabbi, is based on resolving an optimization problem in each period) which achieves a regret that is logarithmic in $k$ under certain non-degeneracy assumptions on the optimization problem and differentiability assumptions on the valuation distribution. 
In contrast, our constant regret guarantees hold under a finite price menu. 

In Figure \ref{fig:numeric_big} we display the results for the large system. Here, since the DP is intractable, we use the offline benchmark. The resulting regret is negligible relative to the total value as captured by the approximation-factor on the right-hand side of the figure.
We also present the competitive ratio of \off against the full-information benchmark (this upper bounds the competitive ratio of \emph{any non-anticipatory policy}) and observe that is bounded away from $1$, hence showing that the full-information benchmark is $\Omega(T)$ away from the DP in our randomly generated instance, which confirms the need for our refined benchmark.
\subsection{Posted Pricing With Customer Choice}
\label{ssec:consumerchoice}

We now consider settings where customers, rather than requesting a specific product, make a choice between multiple substitutes. 
As a concrete example, consider a hardware store selling washers and dryers; the store can set a separate price for a washer, a dryer, and also for buying a washer-and-dryer bundle (i.e., one of each). An incoming customer sees the prices and chooses to buy each of the three options (or nothing at all) with some probability depending on the price menu.
See~\cite[Chapter 7]{talluri_book} for details on such customer-choice models. For exposition, we focus here on a single-customer-type, with arbitrary (but known) customer-choice model.

As before, the controller chooses a price to post for each product and selling one unit of product $j\in [n]$ depletes resources according to $A_j\in \crl{0,1}^d$. There is a discrete set of ``assortment menus'', denoted by $\Alpha$. 
An assortment $\alpha\in\Alpha$ is associated with a vector of prices $(f_{1\alpha},\ldots,f_{n\alpha})$, one price per product. Setting $f_{j\alpha}=\infty$ corresponds to not offering product $j$. 
Note that if each product's price is restricted to take one of $m$ distinct values, then there are at most $\abs{\Alpha} \leq m^n$ different assortments. The actual number of relevant assortments might, however, be much smaller than this.

An arriving customer, when offered assortment $\alpha$, chooses to buy product $j$ with a probability $p_j(\alpha)$, with $\sum_{j=0}^n p_j(\alpha)=1$ (where we use $j=0$ for the no-purchase option). 
These probabilities might be derived, for example, from a standard family such as the multinomial-logit model, nested logit model, etc.; our results do not need any specific structure on the choice probabilities (although assuming more structure may lead to better regret scaling with respect to the number of price menus and more efficient ways of solving the resulting LP relaxation).

The process unfolds as follows: (i) at time $t$ the controller posts an assortment $\alpha\in \Alpha$; (ii) with probability $p_j(\alpha)$, the arriving customer buys one unit of product $j$ (with product $0$ corresponding to no-purchase). 
Now given the choice probabilities, we can simulate the choice model as follows: we assume w.l.o.g. that the customer arriving at time $t$ is endowed with an i.i.d random variable $\xit\sim \text{Uniform}(0,1)$, and assert that the customer buys product $j$ if $\xit \in $ $[\sum_{j'=0}^{j-1}p_{j'}(\alpha), \sum_{j'=0}^jp_{j'}(\alpha)]$. Note that the order of products here is arbitrary.

Applying \rabbi to this setting gives the following result.
\begin{theorem}[Dynamic Pricing with Customer Choice]
\label{theo:reg_multi_pricing}
For any choice model with probabilities and prices $(p_j(\alpha),f_{j\alpha}:j\in [n],\alpha\in\Alpha)$, \rabbi obtains a regret that depends only on $(A,p,f)$, but is independent of the horizon length $T$ and initial budget levels $B\in\N^d$. 
\end{theorem}

\paragraph{Algorithm and Analysis:} The following LP extends~\cref{eq:pricing_lp} to incorporate consumer choice.
\begin{alignat}{2}
\label{eq:choosing_lp}
(P[\vb,\vq,\vz])\qquad & \text{maximize: } & & \sum_{\alpha\in\Alpha}x_\alpha\sum_{j\in [n]}\f_{j\alpha}q_{j\alpha} \\
& \text{subject to: }& \quad & \begin{aligned}[t]
\sum_{\alpha\in\Alpha}\sum_{j\in [n]}a_{ij}q_{j\alpha}x_{\alpha} &\leq b_i & \forall\,i & \in[d] \nonumber\\
\sum_{\alpha\in\Alpha}x_{\alpha} & = t  \nonumber\\
\vx & \geq 0 \nonumber
\end{aligned}
\end{alignat}

Here $q_{j\alpha}$ stands for the fraction of customers that would buy product $j$ if presented with the price assortment $\alpha$. \rabbi re-solves, in each period, this LP with the expected fraction $q_{j\alpha}=p_j(\alpha)$. 
In contrast, \off knows $Q_{j\alpha}(t)$, the realized fraction of customers that, given assortment $\alpha$, would buy product $j$ (formally, \off is equipped with the canonical augmented filtration with variables $(Q_{j\alpha}(T):j\in [n],\alpha\in\Alpha)$), and solves \cref{eq:choosing_lp} with $q_{j\alpha}=Q_{j\alpha}(t)$, where :  
\begin{equation*}
Q_{j\alpha}(t) \defeq \frac{1}{t}\sum_{\tau=1}^tY_{j\alpha}^t \quad \text{ where }\quad
Y_{j\alpha}^t \defeq 
\In{\sum_{j'=0}^{j-1}p_{j'}(\alpha)\leq \xit \leq \sum_{j'=0}^jp_{j'}(\alpha)}.
\end{equation*}

With the (re)defined key ingredients---namely the LP in \cref{eq:choosing_lp} and \off's information structure---it is evident that that the analysis of this expanded model is identical to that of the basic (no-choice) pricing setting with obvious changes. For example, if assortment $\alpha$ is posted at time $t$, the random collected reward is $\sum_jY_{j\alpha}^tf_{j\alpha}$ and the random inventory at $t-1$ is $b-\sum_jA_jY_{j\alpha}^t$. In turn, the Bellman loss in~\cref{eq:losspricing}  takes on the form 
\begin{equation}
\label{eq:loss_multi_pricing}
\BL(t+1,\vb,\alpha) = P[\vb-\sum_jA_jQ_{j\alpha}(t+1),Q(t+1),t] - \E_{t+1}[ P[\vb-\sum_jA_jY_{j\alpha}^{t+1},Q(t),t]]
\end{equation}
Now we have a sum over products $j$, but the analysis goes through via linearity of expectations.

\paragraph{Numerical Simulations:}
We demonstrate our algorithm for the following simple choice-model with two resources $(R1,R2)$, and three products $(\{R1\},\{R2\},\{R1,R2\})$ (for example, a hardware store selling washers ($R1$), dryers ($R2$) or washer-and-dryer combos $\{R1,R2\}$). 
The controller has initial inventories of each resource, and can choose among one of 7 price assortments: high and low prices with/without discounts for buying the bundle, and price menus assuming stock-out of either or both resource. The price menus and choice probabilities are detailed in \cref{tab:assortments}. We run \rabbi for this instance while scaling the horizon and initial inventory; see \cref{fig:pricing_choice}.

\begin{table}[h]
\small
\begin{tabular}{cl|ccccccc|}
\multicolumn{1}{l}{}                                &  Products & \multicolumn{1}{l}{High} & \multicolumn{1}{l}{High-discount} & \multicolumn{1}{l}{Low} & \multicolumn{1}{l}{Low-discount} & \multicolumn{1}{l}{Only R2} & \multicolumn{1}{l}{Only R1} & \multicolumn{1}{l|}{Stock-out} \\ 
\cline{2-9} 
\multicolumn{1}{c|}{\multirow{3}{*}{$f_{j\alpha}$}} & \{R1\} & 5                        & 5                                 & 3                       & 3                                & $\infty$                               & 5                            & Out                          \\
\multicolumn{1}{c|}{}                               & \{R2\}    & 5                        & 5                                 & 3                       & 3                                & 5                                 & $\infty$                          & $\infty$                          \\
\multicolumn{1}{c|}{}                               & \{R1,R2\}     & 10                       & 9                                 & 6                       & 5                                & $\infty$                               & $\infty$                          & $\infty$                          \\ 
\cline{2-9} 
\multicolumn{1}{c|}{\multirow{3}{*}{$p_j(\alpha)$}} & \{R1\} & 0.2                      & 0.2                               & 0.3                     & 0.3                              & 0                                 & 0.2                          & 0                            \\
\multicolumn{1}{c|}{}                               & \{R2\}    & 0.2                      & 0.2                               & 0.3                     & 0.3                              & 0.2                               & 0                            & 0                            \\
\multicolumn{1}{c|}{}                               & \{R1,R2\}     & 0.1                      & 0.15                              & 0.2                     & 0.25                             & 0                                 & 0                            & 
0            \\
\cline{2-9}                
\end{tabular}
\caption{Example with seven assortments:
We consider high/low prices with and without bundling discount (i.e., buying $\{R1,R2\}$ is cheaper than buying each individually). The other assortments can be used if items sell out.
}
\label{tab:assortments}
\end{table}

\begin{figure}[!ht]
	\centering
	\includegraphics[scale=0.7]{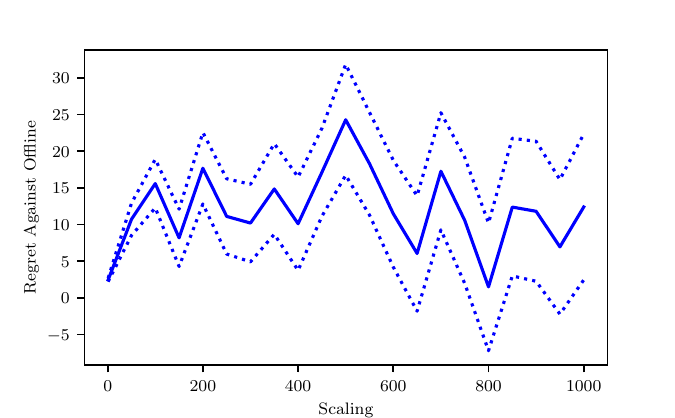}%
	\includegraphics[scale=0.7]{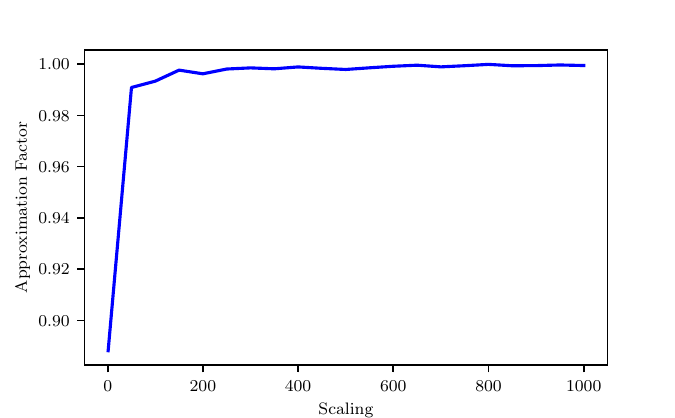}
	\caption{Performance of pricing-\rabbi with customer choice (see~\cref{tab:assortments}). We set the horizon as $T=10k$ and the inventory as $(R1,R2) = (3k,2k)$, and vary scaling parameter $k=1,\ldots,1000$. (LEFT) Regret against \off, with 90\% confidence intervals. (RIGHT) Approximation ratio of \rabbi against DP. 
		\label{fig:pricing_choice}}
\end{figure}

\section{Online Knapsack With Distribution Learning}
\label{sec:learning}

Finally we consider the distribution-agnostic online knapsack setting. We study first the full feedback setting and in \cref{sec:bandits} extend to censored feedback.
As in the baseline \oknap, at each time $t$, the arrival is of type $j\in[n]$ with known probability $p_j$. 
Type $j$ has a known weight $w_j$ and random reward $R_j$, drawn from a distribution $F_j$ with $r_j\defeq \E[R_j]$. 
Critically, we assume $r_j$ and $F_j$ are {\em unknown} to \onl. 

The reward $R_j$ is revealed only after the decision of accept/reject has been made.
At the end of each period, we observe the realization of both accepted and rejected items. 
In contrast, \off has access to the distribution $F_j$, \emph{but not to the realizations}. 
We assume that, before the process starts, we are given one sample of each type, and with $t$ periods to go, define $\Rt_j$ to be the empirical average of the observed rewards for type-$j$ arrivals.

As in probing, we divide each period $t\in\{T,T-1,\ldots,1\}$ into two stages, $t$ and $t-1/2$.
In the first stage (i.e., period $t$) the input reveals the type $j\in[n]$, and in second stages (i.e., period $t-1/2$) the reward is revealed. 
The random inputs are given by $\xit\in[n]$ and $\xi^{t-1/2}\in\R$.
The state space is $\S =\Rp\times\crl{\varnothing,\accept,\reject}$, where the first component is the remaining knapsack capacity. 
At a first stage, given a state of the form $s=(b,\varnothing)$, we choose action $\diamond\in\crl{\accept,\reject}$, reducing the capacity if $\diamond=\accept$.
At the second stage, the state is of the form $s=(b,\diamond)$ with $\diamond\in\{\accept,\reject\}$, and we collect the reward only if $\diamond=\accept$.
Formally, the rewards are $\Re((b,\accept),\xi^{t-1/2},\varnothing)=\xi^{t-1/2}$ and $\Re((b,\reject),\xi^{t-1/2},\varnothing)=0$.

\subsection{Offline Benchmark and Online Policy for Distribution-Agnostic Online Knapsack}

To define \off, $\varphi$ and $\hatphi$, consider the following LP parametrized by $(b,\vy,\vz)\in \Rp\times\R^n\times\Rp^n$
\begin{alignat}{2}
\label{eq:knapsack_lp}
(P[b,\vy,\vz]) \qquad & \text{maximize: } & & \sum_{j}y_{j}x_j  \\
& \text{subject to: }& \quad & 
\begin{aligned}[t]
\sum_{j} w_jx_{j\accept}  &\leq b \nonumber\\
x_{j\accept}+x_{j\reject} & = z_j & \forall\,j & \in [n] \nonumber\\
\vx & \geq 0 \nonumber
\end{aligned}
\end{alignat}
Note that if the average rewards $\vr$ were known, then setting $\vy=\vr$ we get the LP relaxation of \cref{eq:knapsack_val} for the baseline \oknap.
Moreover, for any $\vr$, the optimal LP solution sorts types by their ``bang for the buck'' ratios $r_j/w_j$, and accepts them greedily.
In particular, the solution only requires knowing the ranking induced by $\vr$.

\paragraph{\off Benchmark and Relaxed Value Function:} In this setting, we define \off as the controller that knows the number of arrivals $Z_j^{T}$ for each $j$, and also, \emph{knows the ranking of the types} (i.e., knows $r_j/w_j\,\forall\,j\in[n]$).

Formally, \off is defined via the filtration $\calG_t=\sigma(\crl{\xit:t\in[\T]}\cup\crl{\xi^\tau:\tau\geq t})$.
This is a canonical filtration (see \cref{def:canonical_filtration}) with variables $(G_\theta:\theta\in\Theta)=(\xi^{t}:t\in[\T])$.
Observe that the future rewards, corresponding to times $t-1/2$, are not revealed.
Moreover, the relaxed value is defined as $\varphi(t,s|\calG_t) = P[b,\vr,Z^{t}]$ for first stages and
\begin{equation}\label{eq:knapsack_relax}
\varphi(t-1/2,s|\Gt) = \left\{ 
\begin{array}{ll}
P[b,\vr,Z^{t-1}] & \diamond = \texttt{r} \\ 
\xi^{t-1/2} + P[b,\vr,Z^{t-1}] & \diamond = \texttt{a}.  
\end{array} 
\right.
\end{equation}

\begin{remark}[MDP relaxations for distribution-agnostic settings] 
We note here that the underlying problem in this setting \emph{does not directly admit an MDP}, as the distribution of rewards is unknown. However, once we reveal the arrivals to \off, the relaxation does admit a well-defined MDP. 
By benchmarking against \off, we bypass the need to explicitly formulate an \onl control problem with distribution learning in this setting. 
\end{remark}

\paragraph{Value Function estimate and Online Policy:} 
Recall we define $\Rt_j$ to be the empirical average of the observed rewards for type-$j$ with $t$ periods to go. We define the estimated value as $\hatphi = P[\Bt,\Rt,\E[Z^{t}]]$, resulting in the corresponding online policy given in~\cref{alg:learning}.
\begin{algorithm}
	\caption{Learning \rabbi}
	\label{alg:learning}
	\begin{algorithmic}[1]
		\Require Access to solutions of $(P[b,\vy,\vz])$
		\Ensure Sequence of decisions for \onl.
		\State Set $B^{\T}\gets B$ as the given initial state and $R^T$ as the single sample of each $j$.
		\For{$t\in\{T,T-1,
			\ldots,1\}$}
		\State Compute $\Xt$, an optimal solution to $(P[\Bt,\Rt,\E[Z^{t}]])$.
		\State Observe the arrival type (context), say $\xit=j$, and take any action $\hUt\in\argmax_{u=\accept,\reject}\crl{\Xt_{ju}}$
		\State If $\hUt=\accept$, collect random reward $R_j$ and reduce the budget $B^{t-1}\gets \Bt-w_{j}$. Else, $B^{t-1}\gets \Bt$. 
		\State Update empirical averages $R^{t-1}$ based on $\Rt$ and the observation $R_j$.
		\EndFor
	\end{algorithmic}
\end{algorithm}

\subsection{Regret Analysis for Distribution-Agnostic Online Knapsack}

As in the earlier sections, we first demonstrate that $\varphi$ satisfies the Bellman inequalities
\begin{lemma}\label{lem:relax_knapsack}
The relaxation $\varphi$ defined in \eqref{eq:knapsack_relax} satisfies the Bellman Inequalities with exclusion sets 
\begin{align*}
\calB(t,b)=\crl{\omega\in\Omega:\not\exists X \text{ solving } (P[b,\vr,Z^{t}]) \text{ s.t. }  X_{\xit,\accept}\geq 1 \text{ or } X_{\xit,\reject}\geq 1 }.
\end{align*}
\end{lemma}
\begin{proof}
The initial ordering in \cref{def:relaxed_bellman} follows from an argument identical to that of \cref{lem:probing_initial}.
The monotonicity property follows from \cref{prop:sufficient}.
\end{proof}

To complete the proof of \cref{theo:reg_learning}, we need to characterize the information loss under~\cref{alg:learning}.
The relaxation relies on the knowledge of $\vr$ (the true expectation) and $Z^{t}$.
The natural estimators are the empirical averages $\Rt$ and expectation $\mu(t)=\E[Z^{t}]$, respectively.
Specifically, we use maximizers $\Xt$ of $(P[b,\Rt,\mu(t)])$ to ``guess'' those of $(P[b,\vr,Z^{t}])$.

The overall regret bound is $\rmax(\reg_1+\reg_2)$, where $\reg_1$ and $\reg_2$ are two specific sources of error.
When the estimators $\Rt$ of $\vr$ are accurate enough, the error is $\reg_1$ and is attributed to the incorrect ``guess'' of a satisfying action, i.e., $\reg_1$ is an algorithmic regret. 
The second term, $\reg_2$, is the error that arises from  insufficient accuracy of $\Rt$, i.e., $\reg_2$ is the learning regret. 
The maximum loss satisfies $\rmax\leq \max_{j,i}\crl{w_ir_j/w_j-r_i}$ and we can show that 
\begin{align*}
\reg_1 \leq  2\sum_j \frac{(\wmax/w_j)^2}{p_j} 
\quad \text{ and }\quad 
\reg_2\leq 16\sum_j\frac{1}{p_j(w_j\delta)^2}.
\end{align*}
In sum, the regret is bounded by $(\max_{j,i}\crl{w_ir_j/w_j-r_i})\cdot (2\sum_j \frac{(\wmax/w_j)^2}{p_j} +
 16\sum_j\frac{1}{p_j(w_j\delta)^2})$.
 
\begin{remark}[non-i.i.d arrival processes]\label{rem:arrival_learning}
We used the i.i.d.\ arrival structure to bound two quantities in the proof of \cref{theo:reg_learning}: (1) $\Pr[\norm{Z^{t}-\E[Z^{t}]}\geq c\E[Z^{t}]]$ and (2) $\E[e^{-c\Nt_j}]$, where, recall, $\Nt_j$ is the number of type-$j$ observations. 
The result holds for other arrival processes that admit these tail bounds.
\end{remark}

\subsection{Censored Feedback} \label{sec:bandits}

We consider now the case where only accepted arrivals reveal their reward. 
We retain the assumption of \cref{theo:reg_learning} that there is a separation $\delta>0$: $\abs{\bar r_{j}-\bar r_{j'}}\geq \delta$ for all $j\neq j'$, where $\bar r_j= \E[R_j]/w_j$. 

In the absence of full feedback, we will introduce  a unified approach to obtaining the optimal regret (up to constant factors), that takes the learning method is a plug-in. The learning algorithm will decide between explore or exploit actions.  
Examples of learning algorithms, that also give bounds that are explicit in $t$, include modifications of UCB \citep{srikant}, $\varepsilon$-Greedy or simply to set apart some time for exploration (see \cref{cor:learning} below). 

Recall that $\sigma:[n]\to [n]$ is the ordering of $[n]$ w.r.t.\ the ratios $\bar r_j=r_j/w_j$ and $\sigmat:[n]\to [n]$ is the ordering w.r.t.\ ratios $\bRt_j= \Rt_j/w_j$. The discrepancy $\Pr[\sigma\neq\sigmat]$ depends on the plug-in learning algorithm (henceforth \bandits).
\bandits receives as inputs the current state $\St$ (remaining capacity), time, and the natural filtration $\calF_t$.
The output of \bandits is an action in $\crl{\explore,\exploit}$.
If the action is $\explore$, we accept the current arrival in order to gather information, otherwise we call our algorithm to decide, as summarized in \cref{alg:bandits}. 
Note that $\calF_t$ has information only on the observed rewards, i.e., accepted items. 

\begin{algorithm}
\caption{Bandits \rabbi}
\label{alg:bandits}
\begin{algorithmic}[1]
\Require Access to \bandits and \cref{alg:learning}.
\Ensure Sequence of decisions for \onl.
\State Set $S^\T$ as the given initial state
\For{$t=\T,\ldots,1$}
	\State Observe input $\xit$ and let $U\gets\bandits(T,t,\St,\calF_t)$.
	\State If $U=\explore$, accept the arrival
	\State If $U=\exploit$, take the action given by \cref{alg:learning} 
	\State Update state $S^{t-1}\gets \St-w_{\xit}$ if accept or $S^{t-1}\gets \St$ if reject.
\EndFor
\end{algorithmic}
\end{algorithm}

\begin{theorem}\label{theo:reg_sum}
Let $\reg_1$ be the regret of \cref{alg:learning}, as given in \cref{theo:reg_learning}.
Define the indicators $\explore_t,\exploit_t$ which denote the output of \bandits at time $t$.
The regret of \cref{alg:bandits} is at most $\rmax M$, where
\begin{align*}
M =
\reg_1 + \E\brk*{\sum_t\explore_t} + \E\brk*{\sum_t\Pr[\sigma\neq\sigmat]\exploit_t}.
\end{align*}
\end{theorem}
The expected regret of \cref{alg:bandits} is thus bounded by the regret of \cref{alg:learning} in the full feedback setting, plus a quantity controlled by \bandits. 
In the periods where \bandits says \explore (which, in particular, implies accepting the item), the decision might be the wrong one (i.e., different than \off's). We upper bound this by the number of exploration periods. This is the second term in $M$. The decision might also be wrong if \bandits says \exploit (in which case we call \cref{alg:learning}), but the (learned) ranking at time $t$, $\hat{\sigma}^t$, is different than $\sigma^t$. This is the last term in $M$. 
Finally, even if the learned ranking is correct, \exploit can lead to the wrong ``guess'' by \cref{alg:learning} because the arrival process is uncertain. 
This is the first term in $M$.

\cref{cor:learning} uses a naive \bandits which explores until obtaining $\Omega(\log\T)$ samples and achieves the optimal (i.e., logarithmic) regret scaling.  
The constants may be improved by changing the  \bandits module we use; any such algorithm has the guarantee given by \cref{theo:reg_sum}.
With the naive \bandits, the bound follows from a generalization of coupon collector~\citep{coupon}.

\begin{corollary}\label{cor:learning}
If we first obtain $\frac{8}{(w_j\delta)^2}\log\T$ samples of every type $j$, then we can obtain $O(\log\T)$ regret, which is optimal up to constant factors.
\end{corollary}

\section{Concluding Remarks}
We developed a framework that provides rigorous support to the use of simple optimization problems as a basis for online re-solving algorithms. 
The framework is based on using a carefully chosen offline benchmark, that guides the online algorithm.
The regret bounds then follow from our use of Bellman Inequalities and a useful distinction between Bellman Loss and Information Loss.

As is often the case in approximate dynamic programming, the identification of a function $\varphi$ satisfying the Bellman Inequalities requires some ad-hoc creativity but, as our example illustrate, is often rather intuitive. In \cref{sec:howto} we provide sufficient conditions, applicable to cases where $\varphi$ has a natural linear representation, to verify the Bellman inequalities. 
These conditions are intuitive and likely to hold for a variety of resource allocation problems. Importantly, once such a function is identified, our \rabbi framework provides a way of obtaining online policies from $\varphi$, and corresponding regret bounds. 

We illustrate our framework on three settings. First we consider online probing, which serves as an instance of a larger family of two-stage decision problems, wherein there is an inherent trade-off between getting refined information, and the cost of obtaining it. Next we consider dynamic pricing, which is a well-studied problem, and is representative of settings where rewards and transitions are random.
Finally, our study of online contextual bandits with knapsacks showcases a separation of the underlying combinatorial problem from the parameter estimation problem.  

It is our hope that this structured framework will be useful in developing online algorithms for other problems, whether these are extensions of those we studied here or completely different.

\ACKNOWLEDGMENT{SB and AV gratefully acknowledge support from the ARL under grant W911NF-17-1-0094, and the NSF under grants CNS-1955997, DMS-1839346 and ECCS-1847393; IG's work was supported by the DoD under grant W911NF-20-C-0008.}




\newpage
\begin{APPENDICES}{}

\section{A Sufficient Condition for Bellman Inequalities}
\label[appendix]{sec:howto}

In this section, we construct $\varphi$ based on a general optimization program and provide a sufficient condition to guarantee monotonicity. 
This serves to underscore some of the key elements in a problem's structure that allows one to construct low regret online policies. 
This guideline does not apply to all the examples we study here: in particular, it applies to the baseline and learning variants, but not to probing or pricing.

We study a particular case of canonical filtrations (see \cref{def:canonical_filtration}),  where the random variables $G_\theta$ that we reveal are the inputs $\xi^\theta$ for some fixed times $\Theta$ (see \cref{fig:canonical} for an illustration).

Recall that we reveal some inputs to \off, but not necessarily all of them; we call \emph{concealed inputs} those not revealed to \off.
Informally speaking, we will show that $\varphi$ satisfies the Bellman Inequalities  if (i) \off's relaxed value $\varphi$ can be computed with a linear program and (ii) the concealed inputs are in the objective function only (not in the constraints). 
Requirements (i) and (ii) are appealing because they are verifiable directly from the problem structure without any computation.

Recall that, with $t$ periods to go, \off knows the randomness $\crl{\xi^T,\ldots,\xit,\xi_\Theta}$, where we denote $\xi_\Theta = (\xi^\theta:\theta\in\Theta)$.
In other words, we reveal $\crl{\xi^T,\ldots,\xit,\xi_\Theta}$, while the inputs $\crl{\xi^l:l<t, l\notin\Theta}$ are concealed.

Suppose the relaxation is an LP with decision variables $\vx$ (see \cref{eq:generic_relaxation}):
\begin{equation}\label{eq:varphilinear} 
\varphi(t,s|\calG_t) = \max_{\vx \in \R^{\Xi\times[T]\times\U}}\crl{\E[h(\vx;\xi^1,\ldots,\xiT)|\calG_t]: g(\vx;s,\xi^T,\ldots,\xit,\xi_\Theta)\geq \vec 0},
\end{equation}
where $\U,\Xi$ are the control and input spaces.
For input $\xi$, control $u$, and time $t$, we interpret $x_{\xi,t,u}$ as a variable indicating if \off uses $u$ at time $t$ when presented input $\xi$. 

\begin{proposition}\label{prop:sufficient}
Let $h,g$ be linear functions and let $\varphi$ be given by  \eqref{eq:varphilinear}.
Assume further that the following holds for all $s,t,u$
\begin{itemize}
    \item[(i)] $h$ captures rewards: $\E[h(\ve_{\xi^t,t,u};\xi^1,\ldots,\xiT)|\calG_t]\leq \Re(s,\xi^t,u)$ for actions $u$ that are feasible in state $s$. 
    \item[(ii)] $g$ captures transitions: $g(\ve_{\xi^t,t,u};s,\xi^T,\ldots,\xit,\xi_\Theta)\leq g(\vec 0;\Tr(s,\xit,u),\xi^T,\ldots,\xi^{t-1},\xi_\Theta)$.
\end{itemize}
Then, $\varphi$ satisfies monotonicity with exclusion sets 
\begin{align*}
\calB(t,s)= \crl{\omega\in\Omega: \not\exists X[\omega] \text{ solving } \varphi(t,s|\calG_t) \text{ s.t. } X_{\xit,t,u} \geq 1 \text{ for some }u\in\U}.
\end{align*}
\end{proposition}

It is natural to say that $h$ captures the reward if the incremental effect of taking the action $u$ given input $\xit$ is equal to the immediate reward $\E[h(\ve_{\xi^t,t,u};\xi^1,\ldots,\xiT)|\calG_t]=\Re(s,\xi^t,u)$. 
It is similarly natural to say that $g$ captures transitions if it is {\em stable under the one-step transition}, namely, that $g(\ve_{\xi^t,t,u};s,\xi^T,\ldots,\xit,\xi_\Theta)= g(\vec 0;\Tr(s,\xit,u),\xi^T,\ldots,\xi^{t-1},\xi_\Theta)$; in other words, this means that taking the action $u$ at time $t$, has the same effect as taking no action at the state $\Tr(s,\xi^t,u)$. 
This should hold in any reasonable resource consumption problem, e.g., consuming $1$ with $B$ units of budget remaining is the same as not consuming anything with $B-1$ units. 
In the result below we make the weaker assumption that these relationships hold as inequalities.

The baseline and learning variants are useful illustrations of \cref{prop:sufficient}.

\begin{example}[Baseline] 
Let $\calG$ be the full information filtration ($\Theta=[T]$). 
In \cref{sec:framework} we introduced a linear relaxation for \off. We start by writing a relaxation in the form of \cref{prop:sufficient} and show how it subsequently simplifies to the final form in \cref{sec:framework}. 

Recall that $\accept,\reject$ denote the actions accept and reject. A natural ``expanded'' linear program is
\begin{align*}
\max\crl*{\sum_j\sum_{l=1}^tx_{j,l,\accept}r_j: \sum_{j,l}w_jx_{j,l,\accept}\leq s, 0\leq x_{j,l,\accept}\leq \In{\xi^l=j}}.
\end{align*}
Defining the auxiliary variables $x_j\defeq \sum_lx_{j,l,\accept}$, this is equivalent to  $\varphi(t,s|\calG)=\max\crl{\vr'\vx:\vw'\vx\leq s, 0\leq \vx\leq Z^{t}}$, where, recall $Z_j^t=\sum_{l=1}^t\In{\xi^l=j}$ counts the number of type-$j$ arrivals in the last $t$ periods. 

This $\varphi$ also has the form of \cref{prop:sufficient}, with the functions $h$ and $g$ given by (note that the action \reject has zero objective coefficient)
\begin{align*}
h(\vx;\xi^1,\ldots,\xit) \defeq \sum_jx_{j\accept}r_j
\quad \text{ and } \quad
g(\vx;s,\xiT,\ldots,\xi^1) \defeq 
\begin{pmatrix}
s-\sum_jx_{j\accept} \\
Z^{t}-\vx
\end{pmatrix}.
\end{align*}
Conditions (i) and (ii) can be easily verified now. The objective $h$ is a linear function of the decision vector $\vx$ and the constraint function $g$ aggregates $\xi$ into the sums $Z^{t}$.
\end{example}

In the learning setting, \off is presented with a public type $j$ and must decide whether to accept or reject before seeing the private type, which is a reward $R_j$ drawn from an unknown distribution.

\begin{example}[Learning]
Let us model the problem with $2T$ time periods, where at even times the public type is revealed and at odd times the private (reward).
In this model, the input $\xit$ is an index $j\in[n]$ at even times and it is a reward $R\in \R$ at odd times.
Also let us model the random rewards by drawing i.i.d.\ copies $\crl{R_{jt}}_t$ of $R_j$.

Let us endow \off with the information of all even times, i.e., \off knows all the future arriving public types.
Specifically, we set $\Theta=\crl{t\in[T]:t\text{ is even}}$ (see \cref{fig:canonical} for a representation of $\calG$).
The realizations $\crl{R_{jt}}_{j,t}$, drawn at times $t\notin\Theta$, are concealed.
The expanded linear program is
\begin{align*}
\max\crl*{\sum_j\sum_{l=1}^tx_{j,l,\accept}\E[R_{j}]: \sum_{j,l}w_jx_{j,l,\accept}\leq s, 0\leq x_{j,l,\accept}\leq \In{\xi^l=j}}.
\end{align*}
As before, we can simplify this LP by aggregating variables, see \cref{sec:learning} for the details.
Here we prefer to study the expanded LP because it exemplifies the conditions in \cref{prop:sufficient}.

The objective function is $h(\vx;\xi^1,\ldots,\xit)=\sum_{j,l}x_{j,l,\accept}R_{j,l}$,  When we take expectations $\E[\cdot|\calG_t]$ we arrive at the expression $\sum_{j,l}x_{j,l,\accept}\E[R_j]$.
The constraint function $g$ is given by the feasibility region of the LP.
Conditions (i) and (ii) of \cref{prop:sufficient} hold with equality.
\end{example}

\begin{proofof}{\cref{prop:sufficient}}
Let $u\in\U$ be such that $X_{\xit,t,u}\geq 1$.
Denote $\theta_t\defeq \crl{l\in [T]: l\geq t}\cup\Theta$, so all the inputs $(\xi^l:l\in\Theta_t)$ are revealed at time $t$ (the rest are concealed).
By \cref{lem:collect},
\begin{align*}
\varphi(t,s|\calG_t) = \E[h(\ve_{\xit,t,u};\xi^1,\ldots,\xiT)|\calG_t] + \max_{ \vx}\crl{\E[h(\vx;\xi^1,\ldots,\xiT)|\calG_t]: g(\vx+\ve_{\xit,t,u};s,(\xi^l:l\in\Theta_t)) \geq \vec 0}.
\end{align*}
Using (i) and (ii) yields
\begin{equation}\label{eq:sufficient}
\varphi(t,s|\calG_t) \leq \Re(s,\xit,u)  + \max_{ \vx}\crl{\E[h(\vx;\xi^1,\ldots,\xiT)|\calG_t]: g(\vx;\Tr(s,\xit,u),(\xi^l:l\in\Theta_{t-1}))\geq \vec 0}.
\end{equation}
Since $\calG_t$ is coarser than $\calG_{t-1}$, we know that $\E[\E[\cdot |\calG_{t-1}]|\calG_t]=\E[\cdot|\calG_t]$.
Using \cref{eq:sufficient} and applying Jensen's Inequality (recall that the maximum of linear functions is a convex function) we obtain
\begin{align*}
\varphi(t,s|\calG_t) \leq \Re(s,\xit,u)  + \E\brk*{\max_{ \vx }\crl{\E[h(\vx;\xi^1,\ldots,\xiT)|\calG_{t-1}]: g(\vx;\Tr(s,\xit,u),(\xi^l:l\in\Theta_{t-1}))\geq \vec 0}\Big|\calG_t}.
\end{align*}
This corresponds to the required inequality in \cref{def:relaxed_bellman}.
\end{proofof}

The sufficient conditions in \cref{prop:sufficient} are not necessary; they are not satisfied in the probing setting (\cref{sec:probing}) or in the pricing setting (\cref{sec:pricing}). 
Nevertheless, we are still able to show monotonicity and draw the desired regret bounds.  
\section{Additional Details from \texorpdfstring{\cref{sec:probing}}{Online Probing} (Online Probing)}
\label[appendix]{sec:probingproofs}

We first state and prove an auxiliary lemma which we need for our proofs.
\begin{lemma} 
\label{lem:concave} 
	Consider the standard-form LP $(P[\vd]):\max\{\vr'\vx:M\vx=\vd,\vx\geq 0\}$, where $M\in\R^{m\times n}$ is an arbitrary constraint matrix and $\vd\in\R^m$. 
	The function $\vd\mapsto P[\vd]$ is concave and therefore, if $X$ is a random right-hand side, then $\E[P[X]]\leq P[\E[X]]$. 
\end{lemma} 
\begin{proof}
	The dual problem is $(D[\vd]):\min\crl{\vd'\vy:M'\vy \geq \vr}$.
	The function $\vd\mapsto D[\vd]$ is a minimum of linear functions, therefore concave.
\end{proof}

\subsection{Bellman Inequalities and Loss}

We first establish the initial ordering property.

\begin{proofof}{\cref{lem:probing_initial}}
Consider a policy for \off determining when to probe, accept or reject.
Recall such a policy is a mapping $\pi:[T]\times\S\to\U$ s.t. $\pi(t,s)$ is $\calG_t$-measurable for all $t,s$. 
	
The policy, once fixed, induces a random trajectory determined by the realization of the probed rewards. 
Denote the random number of times where a type $j$ was probed as $X_{j\probe}$, accepted (rejected) without probing as $X_{j\accept}$ ($X_{j\reject}$), and accepted (rejected) after probe outcome is $(j,k)$ as $X_{jk\accept}$ ($X_{jk\reject}$). 
Then, we can write $\E[V(\T,(b_h,b_p)|\calG_{T})] = \E[\sum_j \bar r_jX_{j\accept} + \sum_{j,k}r_{jk}X_{jk\accept} |\calG_{T}]$, where we use the fact that, conditional on accepting without probing, the expected reward is $\bar r_j$. 
Thus we have
\begin{align*}
	\E[V(t,(b_h,b_p)|\calG_{T})] = \sum_j\bar r_j\E[X_{j\accept}|\calG_{T}] + \sum_{j,k} r_{jk}\E[X_{jk\accept}|\calG_{T}].
\end{align*} 
	
We now claim that $\E[\mathbf{X}]$ yields a feasible solution to $(P[T,(b_h,b_p),Z])$. 
Indeed, with the exception of the constraint $x_{jk\accept}+x_{jk\reject} = q_{jk}x_{j\probe}$, the random variables satisfy a.s.\ all the constraints of $(P[T,(b_h,b_p),Z])$.
	Furthermore, since \off's policy is adapted to $\calG$, we obtain $\E[X_{jk\accept}+X_{jk\reject}|X_{j\probe},\calG_{T}]=q_{jk}X_{j\probe}$, thus the expected values satisfy the desired constraint.
	To summarize, $V(T,(b_h,b_p)|\calG_{T})$ equals the value of the feasible solution given by the expectations.
\end{proofof}

Next we establish the monotonicity condition in \cref{def:relaxed_bellman}.

\begin{proof}[Proof of~\cref{lem:probing_ineq}] 
Observe that the monotonicity condition in~\cref{def:relaxed_bellman} translates to the following condition in the online probing setting.
\begin{align*}
\varphi(t,(b_h,b_p,\varnothing)|\calG_t) \leq \max_{\diamond\in\crl{\accept,\probe,\reject}}\crl{\E_{\xi^{t-1/2}}[\varphi(t-1/2,(s_{\diamond},\diamond)|\calG_{t-1/2})|\calG_t]} \quad \forall \omega\notin\calB(t,s).
\end{align*}
where the state $s_{\diamond}=(b_h-1,b_p)$ if $\diamond=\accept$, $s_{\diamond}=(b_h,b_p-1)$ if $\diamond=\probe$ and $s_{\diamond}=(b_h,b_p)$ if $\diamond=\reject$. 

First, given $\xit=i$, we have from~\cref{eq:probing_relax} that $\E_{\xi^{t-1/2}}[\varphi(t-1/2,(s_{\diamond},\diamond))|\calG_{t-1/2})|\calG_t]=P[(b_h,b_p),Z^{t-1}] $ if $\diamond = \reject$, and $r_{\xi^{t-1/2}}+P[(b_h-1,b_p),Z^{t-1}] $ if $\diamond = \accept$. 
Now for cases (1) and (2), the claim in the lemma follows directly by invoking~\cref{lem:collect}.

For case (3) we need to introduce some notation.
Let $\vec q_{j}\in\R^{n\times m}$ be a vector with value $q_{jk}$ in components $(j,k)$, $k\in[m]$, and zero otherwise (i.e.\ in components $(j',k)$ with $j'\neq j$).
Similarly, let $\vec 1_{(j,k)}\in\R^{n\times m}$ have value $1$ in the single component $(j,k)$ and zero otherwise.
We also rewrite the LP in~\cref{eq:probing_lp} with an extra `budget vector' $\vw$ such that $P[(b_h,b_p),\vz]=\bar P[(b_h,b_p),\vz,\vec 0]$.
	\begin{alignat*}{2}
	\label{eq:probing_lp1}
	(\bar P[(b_h,b_p),\vz,\vw]) \qquad & \text{maximize: } & & \sum_{j,k}r_{jk}x_{jk\accept}+ \sum_j \bar{r}_jx_{j\accept} \\
	& \text{subject to: }& \quad & 
	\begin{aligned}[t]
	\sum_{j,k}x_{jk\accept}+\sum_jx_{j\accept} & \leq b_h \nonumber\\
	\sum_{j}x_{j\probe} & \leq b_p \nonumber\\
	x_{j\accept} +x_{j\probe} + x_{j\reject} & = z_j& \forall\,j & \in [n] \nonumber\\
	x_{jk\accept} +x_{jk\reject} - q_{jk}x_{j\probe}& = w_{jk} & \forall\,j & \in [n], k \in [m] \nonumber\\
	\vx & \geq 0 \nonumber
	\end{aligned}
	\end{alignat*}
	Now if $\bar{X}_{i\probe}\geq 1$ and $\xi^{t-1/2}=(i,k)$ is such that either $\bar{X}_{ik\accept}\geq 1$ or $\bar{X}_{ik\reject}\geq 1$, then by \cref{lem:collect}, we have the following decomposition (depending on the random $\xi^{t-1/2}$)
	\begin{align*}
	\bar P[(b_h,b_p),Z^{t},\vec 0] = r_{\xi^{t-1/2}}\In{\bar{X}_{ik\accept}\geq 1} + \bar P[(b_h-\In{\bar{X}_{ik\accept}\geq 1},b_p-1),Z^{t-1},\vec q_{\xit}-\vec 1_{\xi^{t-1/2}}], \quad \forall\omega\notin\calB(t,b_h,b_p)
	\end{align*}
	where the vectors $\vec q,\vec 1$ are evaluated in random components; since by assumption $\bar X_{i\probe}\geq 1$ under the optimal solution, the optimal value in the optimization problem is the same as the reward obtained ``now'' ($r_{\xi^{t-1/2}}$) and the residual value after discounting $b_p$ by one. 
	Taking expectations $\E[\cdot |\calG_t]$ and using \cref{lem:concave} we have 
	\begin{align*}
	P[(b_h,b_p),Z^{t}] &= \E[r_{\xi^{t-1/2}}\In{\bar{X}_{ik\accept}\geq 1} + \bar P[(b_h-\In{\bar{X}_{ik\accept}\geq 1},b_p-1),Z^{t-1},\vec q_{\xit}-\vec 1_{\xi^{t-1/2}}] |\calG_t] \\
	&\leq \E[r_{\xi^{t-1/2}}\In{\bar{X}_{ik\accept}\geq 1}+ \bar P[(b_h-\In{\bar{X}_{ik\accept}\geq 1},b_p-1),Z^{t-1}, \vec 0] |\calG_t]\\
	&\leq \E[\max\crl{r_{\xi^{t-1/2}}+P[(b_h-1,b_p-1),Z^{t-1}],P[(b_h,b_p-1),Z^{t-1}]}|\calG_t].
	\end{align*}
	The last inequality, following from substituting $\In{\bar{X}_{ik\accept}\geq 1}\in\{0,1\}$, gives the desired result.
\end{proof}

\section{Additional Details from \texorpdfstring{\cref{sec:pricing}}{Dynamic Pricing} (Dynamic Pricing)}
\label[appendix]{sec:pricingproofs}

\subsection{Proof of \texorpdfstring{\cref{lem:pricing_ineqs}}{}}

Throughout this subsection, we fix some indexes $j',l'$.
To complete the proof of the proposition, it remains to establish that, whenever $X_{j'l'}\geq 1$, then   $\E[\BL(t+1,\vb,j',l')|\calG_t]\leq 0$, where 
$$    \BL(t+1,\vb,j',l') = P[\vb-A_{j'}Q_{j'l'}(t+1),Q(t+1),Z^{t}] - \E_{t+1}[ P[\vb-A_{j'}Y_{j'l'},Q(t),Z^{t}]].$$

\paragraph{The Correction LP.} Let us fix $(t,\vb,\vq,\vz)$ and denote $\bar \vx$ the solution of $P[\vb-A_{j'}q_{j'l'},\vq,\vz]$. 
To bound the loss, we must bound the right-hand side of \eqref{eq:losspricing}, which captures the perturbation of budgets from $\vb-A_{j'}q_{j'l'}$ to $\vb-A_{j'}Y_{j'l'}$ and the perturbation of fractions from $\vq$ to $\vq+\Delta$, where $\Delta$ is a zero-mean random vector. 

Let us re-formulate $P[\vb-A_{j'}Y_{j'l'},\vq+\Delta,\vz]$ based on how much we need to correct $\bar \vx$:
\begin{equation*}
\begin{array}{rrll}
(P[\vb-A_{j'}Y_{j'l'},\vq+\Delta,\vz]) \qquad \max_{\vy} & \qquad \sum_{j,l}\f_{jl}(q_{jl}+\Delta_{jl})(\bar x_{jl}-y_{jl}) \medskip\\
\text{s.t.}& 
\qquad \sum_{j,l}a_{ij}(q_{jl}+\Delta_{jl})(\bar x_{jl}-y_{jl}) &\leq b_i-a_{ij'}Y_{j'l'} & \qquad \forall i \\
&  \sum_l (\bar x_{jl}-y_{jl}) &\leq z_j &\qquad \forall j   \\
&  \bar \vx-\vy &\geq 0.
\end{array}
\end{equation*}
The new formulation uses decision variables $\vy$, which may be negative, and correspond to how much we movement there is from the initial solution $\bar\vx$ to the new one.

Let us denote the resource-slack variables of $P[\vb-A_{j'}q_{j'l'},\vq,\vz]$ by $(s_i\geq 0:i\in[d])$, i.e., $\sum_{j,l}a_{ij}q_{jl}\bar x_{jl} +s_i =b_i-a_{ij'}q_{j'l'}$.
Similarly, let us denote the demand-slack variables by $(u_j\geq 0:j\in [n])$, i.e., $\sum_l\bar x_{jl}+u_j=z_j$.
Using the slack variables, the problem simplifies to \begin{equation}\label{eq:correction1}
\begin{array}{rrll}
P[\vb-A_{j'}Y_{j'l'},\vq+\Delta,\vz] = \sum_{j,l}f_{jl}(q_{jl}+\Delta_{jl})\bar x_{jl} - \min_{\vy} & \qquad \sum_{j,l}\f_{jl}(q_{jl}+\Delta_{jl})y_{jl} \medskip\\
\text{s.t.}& 
\qquad \sum_{j,l}a_{ij}(q_{jl}+\Delta_{jl})y_{jl} &\geq \beta_i  & \qquad \forall i \\
&  \sum_l y_{jl} & \geq -u_j &\qquad \forall j   \\
&  y &\leq \bar\vx,
\end{array}
\end{equation}
where we defined $\beta_i\defeq a_{ij'}(Y_{j'l'}-q_{j'l'})-s_i+\sum_{j,l}a_{ij}\Delta_{jl}\bar x_{jl}$.

Observe that, since $\E[\Delta]=0$, the first term outside the minimization, namely $ \sum_{j,l}f_{jl}(q_{jl}+\Delta_{jl})\bar x_{jl}$, equals $ \sum_{j,l}f_{jl}q_{jl}\bar x_{jl}= P[\vb-A_{j'}q_{j'l'},\vq,\vz]$ in expectation.
The following result readily proves \cref{lem:pricing_ineqs}.

\begin{lemma}[Correction LP]\label{prop:correction}
If we denote $\vq=Q(t+1)$, then the Bellman Loss is bounded by $\E[\BL(t+1,\vb,j',l')]\leq \E[\CP[Y_{j'l'},\vq,\Delta]]$, where $(\CP[Y_{j'l'},\vq,\Delta])$ is the minimization problem in \cref{eq:correction1}.
Furthermore, $\E[\BL(t+1,\vb,j',l')] \leq 0$.
\end{lemma}
\begin{proof}
Recall that $\beta_i=a_{ij'}(Y_{j'l'}-q_{j'l'})-s_i+\sum_{j,l}a_{ij}\Delta_{jl}\bar x_{jl}$ and observe that $\E[\beta_i]\leq 0$ for all $i$.
We will find some deterministic values $c_i$ such that the objective value of $\CP[Y_{j'l'},\vq,\Delta]$ is upper bounded by $\sum_{i}c_i\beta_i$, which proves the result.

We argue the upper bound on $(\CP[Y_{j'l'},\vq,\Delta])$ by bounding the optimal dual solution.
The dual of $\CP[Y_{j'l'},\vq,\Delta]$ is
\begin{align*}
\max_{\mu,\lambda,\theta \geq 0}\crl*{ \beta'\mu-u'\lambda-\sum_{j,l}\bar x_{jl}\theta_{jl} :  (q_{jl}+\Delta_{jl})A_j'\mu+\lambda_j -\theta_{jl}\leq f_{jl}(q_{jl}+\Delta_{jl}) \quad \forall j,l  }
\end{align*}
This problem is the dual of a feasible and finite problem (see \cref{eq:correction1}), hence it has an optimal finite solution and we can bound $\mu_i\leq c_i$ for some deterministic values $c_i$.
The objective value of this maximization problem is upper bounded by $\beta'c$, which proves the result.
\end{proof}

\subsection{Proof of \texorpdfstring{\cref{lem:pricing_dual}}{}} 

Recall that we wish to establish the following: if $\hat{\varphi}$ (used by \onl) has a solution with $x_{jl'}=\max_{l}x_{jl}>>1$, then posting price $\f_{jl}$ is a satisfying action. 
To establish this, it remains to bound the difference between the LP $\SP[v_t,\vq,\vz]$ and its ``perturbed'' version $\SP[V_t,\vq+\Delta\vq,\vz+\Delta\vz]$. 
To that end, we first establish a bound on $v_t-V_t$; see item (i) in the discussion following \cref{lem:sel}.

\begin{lemma}\label{lem:fluid_gap}
For fixed $\vb$, denote $V_t = P[\vb,Q(t),Z^{t}]$ and $v_t = P[\vb,\E[Q(t)],\E[Z^{t}]]$.
If $t\geq c$, then, with probability at least $1-c/t^2$, we have $v_t-V_t\geq -c\sqrt{t\log(t)}$. 
The constant $c$ is independent of $\vb$ and depends on $(\vf,F_1,\ldots,F_n)$ only.
\end{lemma}
\begin{proof}
Set $\vq = Q(t)$, $\vz = Z^{t}$, $\Delta\vq = \E[Q(t)]-Q(t)$, and $\Delta\vz = \E[Z^{t}]-Z^{t}$.
Take $\bar\vx$ to be a solution of $V_t$ and use a correction program analogous to \cref{eq:correction1} to conclude
\begin{equation} \label{eq:correction_fluid}
  \begin{array}{rrll}
v_t = V_t +\sum_{j,l}f_{jl}\Delta q_{jl}\bar x_{jl} - \min_{\vy} & \qquad \sum_{j,l}\f_{jl}(q_{jl}+\Delta q_{jl})y_{jl} \medskip\\
\text{s.t.}& 
\qquad \sum_{j,l}a_{ij}(q_{jl}+\Delta q_{jl})y_{jl} &\geq \beta_i  & \qquad \forall i \\
&  \sum_l y_{jl} +tp_j & \geq \sum_l\bar x_{jl} &\qquad \forall j   \\
&  \vy &\leq \bar\vx,
\end{array}  
\end{equation}
where $\beta_i=-s_i+\sum_{j,l}a_{ij}\Delta q_{jl}\bar x_{jl}$.
We will argue an upper bound on the minimization problem by exhibiting a feasible solution.

Set $g(t)\defeq \sqrt{log(t)/t}$ and consider the solution $y_{jl}=\bar x_{jl} \frac{g(t)}{q_{jl}+\Delta q_{jl}}$.
First recall that, by \cref{lem:dkw}, $\abs{\Delta q_{jl}}\leq g(t)$ with high probability.
The objective value of this solution is $\sum_{j,l}f_{jl}\bar x_{jl}g(t)$, hence from \cref{eq:correction_fluid} we get $v_t\geq V_t -2\sum_{j,l}f_{jl}g(t)\bar x_{jl}$.
From here, using the fact that $\bar \vx$ solves an LP with the constraint $\sum_lx_{jl}\leq Z_j^{t}$ for all $j$ and that $Z_j^{t}\leq t$ a.s., we conclude the result by using that $\bar x_{jl}\leq t$.

We are left to check that our solution $\vy$ is feasible for the LP in \cref{eq:correction_fluid}.
The first set of constraints is satisfied because $g(t)\geq \Delta q_{jl}$.
The second set of constraints is satisfied since $\sum_{l}\bar x_{jl}\leq Z_j^{t}$ and $Z_j^{t}\leq tp_j+\sqrt{t\log (t)}$ w.h.p.
Finally, the constraints $\vy\leq \bar\vx$ are  satisfied since $g(t)\leq q_{jl}+\Delta q_{jl}$ for all $t$ large enough.
\end{proof}

\begin{proofof}{\cref{lem:pricing_dual}}
Let us denote $\theta=(v,\vq,\vz)$ and $\theta+\Delta\theta =(v+\Delta v,\vq+\Delta\vq,\vz+\Delta\vz)$.
Recall that $\vb$ and $t$ are fixed throughout.
The selection program for a fixed component $(j',l')$ is given by
\begin{align*}
\begin{array}{rrll}
(\SP[\theta]) \qquad \max & \qquad x_{j'l'} \\
\text{s.t.}& \qquad \sum_{j,l}f_{jl}q_{jl}x_{jl} &\geq v \\
& \sum_{j,l}a_{ij}q_{jl}x_{jl} &\leq b_i  & \qquad \forall i \\
&  \sum_l x_{jl} & \leq z_j &\qquad \forall j   \\
&  \vx &\geq 0.
\end{array}
\end{align*}

If $\bar X$ is the solution to $P[\vb,Q(t),Z^{t}]$ (used by \off) and $\bar \vx$ is the solution to $P[\vb,\E[Q(t)],\E[Z^{t}]]$ (used by \onl), we want to prove $\bar X_{j'l'}\geq \bar x_{j'l'}-c\sqrt{t\log(t)}$.
Equivalently, our aim is to prove the following:
\begin{align*}
\SP[\theta+\Delta\theta] \geq \SP[\theta] - c\sqrt{t\log(t)} \qquad \text{ w.p. } 1-c/t^2.
\end{align*}

We argue via Lagrangian relaxation.
The Lagrangian of the selection problem with parameters $\theta+\Delta$ is given by
\begin{align*}
 L(\vx,\lambda;\theta+\Delta\theta) &= x_{j'l'} +\lambda_0\prn*{\sum_{j,l}f_{jl}(q_{jl}+\Delta q_{jl})x_{jl}-v-\Delta v}
+\sum_i\lambda_i\prn*{b_i-\sum_{j,l}a_{ij}(q_{jl}+\Delta q_{jl})x_{jl}}\\
&+\sum_j\lambda_j\prn*{z_j+\Delta z_j-\sum_lx_{jl}}  \\
&= L(\vx,\lambda;\theta) +\lambda_0\prn*{\sum_{j,l}f_{jl}\Delta q_{jl}x_{jl}-\Delta v}
-\sum_i\lambda_i\sum_{j,l}a_{ij}\Delta q_{jl}x_{jl}
+\sum_j\lambda_j\Delta z_j
\end{align*}

Define $D\defeq\crl{\vx: \vx\geq 0, \norm{\vx}_\infty\leq t}$.
Observe that both $\SP[\theta]$ and $\SP[\theta+\Delta\theta]$ have solutions $\vx\in D$.
From \cref{lem:dkw} and \cref{lem:fluid_gap} we have the following with probability $1-c/t^2$:
\begin{align*}
\abs{\Delta q_{jl}x_{jl}} \leq  \sqrt{t\log(t)} \quad\forall \vx\in D, \quad
\Delta v \leq \sqrt{t\log(t)}, \quad
\Delta z_j \geq \sqrt{t \log(t)}.
\end{align*}

Let $\lambda^\star$ be the optimal dual variables of $\SP[\theta+\Delta\theta]$.
We claim that there is a constant $c$ such that $\norm{\lambda^\star}_\infty \leq c$.
Assuming this claim, from the previous equation we get
\begin{align*}
 L(\vx,\lambda;\theta+\Delta\theta) \geq L(\vx,\lambda;\theta) -c\sqrt{t \log(t)} \quad \forall \vx\in D.
\end{align*}

Using Strong Duality for the problem $\SP[\theta+\Delta\theta]$ we have
\begin{align*}
\SP[\theta+\Delta\theta] &= \max_{\vx\geq 0}L(\vx,\lambda^\star;\theta+\Delta\theta) \\
&= \max_{\vx\in D}L(\vx,\lambda^\star;\theta+\Delta\theta) \\
&\geq \max_{\vx\in D}L(\vx,\lambda^\star;\theta) -c\sqrt{t \log(t)} \\
& \geq \SP[\theta]  -c\sqrt{t \log(t)}.
\end{align*}
In the last step we used weak duality.
Finally, to bound $\norm{\lambda^\star}_\infty \leq c$ we observe that the dual feasible region is defined by $\lambda\geq 0$ and the following set of inequalities, where $\delta$ is the Kronecker delta:
\begin{align*}
-f_{jl}q_{jl}\lambda_0 +q_{jl}\sum_{i}a_{ij}\lambda_i +\lambda_j \geq \delta_{j'l'} \quad \forall j,l.
\end{align*}
These inequalities are independent of $(t,\vb)$, hence we can bound uniformly the extreme points.
\end{proofof}

\section{Additional Details from \texorpdfstring{\cref{sec:learning}}{Learning Section} (Distribution-Agnostic Knapsack)} 
\label{sec:proofs_learning}

\begin{proofof}{\cref{theo:reg_learning}}
To apply \cref{theo:resolve}, we first bound the measure of the exclusion sets $\calB$ and the ``disagreement'' sets $\calQ$. 
Recall that $\calB(t,b)$ is given in \cref{lem:relax_knapsack} and $\calQ(t,b)$ is the event where $\hUt$ is not a satisfying action.
 
Let $\sigma:[n]\to [n]$ be an ordering of $[n]$ w.r.t.\ the ratios $\bar r_j:=\frac{r_j}{w_j}$ such that $\sigma_j =1 $ if $j$ has the highest ratio. Similarly, let $\sigmat:[n]\to [n]$ be the ordering w.r.t.\ ratios $\bRt_j\defeq \Rt_j/w_j$.

Call $E^t$ the event $\calB(t,\Bt)\cup\calQ(t,\Bt)$, then
\begin{align*}
\Pr[E^t] = \Pr[E^t,\sigma = \sigmat] + \Pr[E^t,\sigma\neq\sigmat]
\leq \Pr[E^t,\sigma = \sigmat] + \Pr[\sigma\neq\sigmat].
\end{align*}

Let $\Nt_j$ be the number of type-$j$ samples observed by the beginning of period $t$.
By definition, since we are given a sample of each type before the process starts, we have $\Nt_j=Z_j^{T}-Z_j^{t}+1$. 
Since the reward distribution is sub-Gaussian, it satisfies the Chernoff bound \citep{concentration_book}
\begin{equation}\label{eq:subgauss}
\Pr[\Rt_j-r_j\geq x|\Nt_j],\Pr[\Rt_j-r_j\leq x|\Nt_j] \leq e^{-\Nt_jx^2/2} \qquad \forall x\in\R,
\end{equation}

A union bound relying on \cref{eq:subgauss} gives that 
\begin{align*}
\Pr[\sigma\neq\sigmat|\calF_t] \leq \Pr[ \exists j \text{ s.t. } \abs{\bar r_j-\bRt_j}\geq \delta/2|\calF_t]
\leq 2\sum_je^{-\Nt_j(w_j\delta)^2/8}. 
\end{align*}
The variable $\Nt_j$, recall, is the number of type-$j$ samples observed by the beginning of period $t$, hence $\Nt_j-1$ is a $\Bin(T-t,p_j)$ random variable. 
It a known fact that, given $\theta>0$, $\mathbb{E}[e^{-\theta\Bin(p,m)}]=(1-p+pe^{-\theta})^m$, thus 
\begin{align*}
\Pr[\sigma\neq\sigmat]=\E[\Pr[\sigma\neq\sigmat|\calF_t]] \leq 2\sum_je^{-(w_j\delta)^2/8}(1-p_j+p_je^{-(w_j\delta)^2/8})^{T-t}. 
\end{align*}
Upper bounding by a geometric sum yields
\begin{equation} \label{eq:reg2}
\reg_2 \defeq
\sum_t \Pr[\sigma\neq\sigmat] \leq 2\sum_j \frac{1}{p_j(e^{(w_j\delta)^2/8}-1)}
\leq 2\sum_j\frac{8}{p_j(w_j\delta)^2}.    
\end{equation}

We are left to bound $\Pr[E^t,\sigma = \sigmat]$.
Let us assume w.l.o.g.\ that the indexes are ordered so that $\bar r_1\geq \bar r_2 \geq \ldots \geq \bar r_n$.
The optimal solution of $(P[\Bt,\vr,Z^{t}])$, i.e., \off's problem, is to sort the items and accept starting from $j=1$, without exceeding the capacity $\Bt$ or the number of arrivals $Z_j^{t}$. 
Mathematically, the optimal solution $\Xts$  to $(P[\Bt,\vr,Z^{t}])$ is
\begin{align*}
\Xts_{1\accept} = \min\crl*{Z_1^t,\frac{\Bt}{w_1}}, \quad 
\Xts_{j\accept} = \min\crl*{Z_j^{t},\frac{\Bt-\sum_{i<j}w_i\Xts_{i\accept}}{w_j}} \quad j=2,\ldots,n.
\end{align*}
For the proxy $(P[\Bt,\Rt,\mu(t)])$, the optimal solution has the same structure with $Z_j^{t}$ replaced everywhere by $\mu_j(t)$.

Let $\xit=j$ and $U$ be any action in $\argmax\crl{\Xt_{j,u}:u=\accept,\reject}$.
We study first the case $U=\accept$. If 
$\Xts_{j,\accept}\geq 1$ then $U=\accept$ would be, by \cref{lem:relax_knapsack}, a satisfying action. 
If it is not a satisfying action it must then be that $\Xts_{j,\accept}<1$ and since the algorithm chooses to accept it must be also that $\Xt_{j,\accept}\geq \mu_j(t)/2$. Thus we obtain the following two conditions
\begin{align*}
\Xts_{j,\accept}<1 \Rightarrow \sum_{i<j}w_iZ_i^t \geq b \quad \text{ and } \quad 
\Xt_{j,\accept}\geq \mu_j(t)/2 \Rightarrow
\sum_{i<j}w_i\mu_i(t) + w_j\mu_j(t)/2\leq b.
\end{align*}

In the case $U=\reject$, $\Xts_{j,\reject}<1$ and $\Xt_{j,\reject}\geq \mu_j(t)/2$ imply
\begin{align*}
\sum_{i\leq j}w_iZ_i^{t} \leq b \quad \text{ and } \quad 
\sum_{i<j}w_i\mu_i(t) + w_j\mu_j(t)/2\geq b.
\end{align*}
In conclusion, 
\begin{align*}
\Pr[E^t,\sigma = \sigmat] \leq 
\max\crl*{ \Pr\brk*{\sum_{i\leq j}w_i(Z_i^{t}-\mu_i(t)) \geq \frac{w_j\mu_j(t)}{2}},  \Pr\brk*{\sum_{i\leq j}w_i(Z_i^{t}-\mu_i(t) )\leq -\frac{w_j\mu_j(t)}{2}} }.
\end{align*}

These probabilities are bounded symmetrically using the method of averaged bounded differences \citep[Theorem 5.3]{dubhashi}.
Indeed, using the natural linear function $f(\xi^1,\ldots,\xit)=\sum_iw_i\sum_{l=1}^t\In{\xi^l=i}$, the differences are bounded by $\abs{\E[f|\calF_l]-\E[f|\calF_{l-1}]}\leq\wmax$, hence 
\begin{align*}
\reg_1 \defeq
\sum_t\Pr[E^t,\sigma = \sigmat] \leq \sum_t\sum_jp_j \exp\prn*{-\frac{2(w_j\mu_j(t)/2)^2}{t\wmax^2}}
\leq 2\sum_j \frac{(\wmax/w_j)^2}{p_j}.
\end{align*}
Together with \cref{eq:reg2}, we have the desired bound.
\end{proofof}

\section{Connections to Information Relaxations}
\label[appendix]{sec:inforelax}

Our work is related to the information-relaxation framework developed in \citep{info_relaxation,info_relaxation2}. The information-relaxation framework is a fairly general way to endow \off with additional information, but at the same time forcing him to pay a penalty for using this information. The dual problem (with the penalties) is an upper bound on the performance of the best online policy. 

The main distinctions with our approach are:
\begin{itemize}
    \item[1.] Information Relaxation requires to identify \off's filtration and penalties to build a proxy for \off's value function. This proxy can then be used to assess the performance of specific online policies. 
    
    The proxy that is developed---as the true \off value in our framework---may be difficult to compute. To overcome this difficulty, \citep{info_relaxation2} proposes an approximation through which penalties can be computed and hence an upper bound can be obtained. 
    
    \item[2.] Our framework requires, as well, identifying a suitable information structure (a filtration) and a relaxation $\varphi$.
    Because we allow for a Bellman Loss, we can develop $\varphi$ $\hatphi$ that are computationally tractable. In most cases, a linear program. The framework explicitly then provides a mechanism, the \rabbi algorithm, to derive a good online policy. 
\end{itemize}

There is also an explicit mathematical connection. To state it, we first present a weaker version of our Bellman Inequalities, called thus because it is easier to find an object $\varphi$ under this definition.
Recall that, for a given non-anticipatory policy $\pi$, we denote $\von_\pi$ the expected value.
Observe that the distinction with \cref{def:relaxed_bellman} is in the initial ordering condition; we now require $\phi$ to upper bound the online value instead of the best offline.

\begin{definition}[Weak Bellman Inequalities]\label{def:weak_bellman}
The sequence of r.v.\ $\crl{\varphi(t,s)}_{t\in\T,s\in\S}$  satisfies the Weak Bellman Inequalities w.r.t.\ filtration $\calG$ and events $\calB(t,s)\subseteq\Omega$ if $\varphi(t,s)$ is $\calG_t$-measurable for all $t,s$ and the following holds:
\begin{enumerate}
\item Initial ordering: $\max_{\pi}\von_\pi \leq \E[\varphi(\T,S^\T|\calG_{\T})]$, where $S^T$ is the initial state.
\item Monotonicity: $\forall s\in\S,t\in[T],\omega\notin\calB(t,s)$, 
\begin{equation}\label{eq:weak_bellman}
\varphi(t,s|\calG_t) \leq \max_{u\in\U}\crl{\Re(s,\xit,u)+\E[\varphi(t-1,\Tr(s,\xit,u)|\calG_{t-1})|\calG_t]}.
\end{equation}
\end{enumerate}
\end{definition}

In Proposition 2.1 in \citep{info_relaxation2} it is shown that if $\varphi$ is some function that satisfies the Bellman equation for \off with the penalized immediate rewards function, then, in particular, it satisfies the initial ordering above. Since such $\varphi$ satisfies, by construction, the Bellman inequality the following is an immediate corollary. 

\begin{proposition}[Proposition 2.1 in \citep{info_relaxation2}]
Given feasible penalties $z_t$, the penalized value function satisfies \cref{def:relaxed_bellman} with exclusion sets $\calB(t,s)=\varnothing$.
\end{proposition}

Our framework is a structured approach for building a computationally tractable $\varphi$, and deriving an online policy is bounded regret, without pre-computing penalties.

\newpage
\section{Parameter for the Pricing Instance}\label{sec:table}

\begin{center}
\begin{adjustbox}{angle=270}
\footnotesize
\begin{tabular}{llllllllllllllllllllll}
                               &                         & \multicolumn{20}{c}{Type $j$}                                                                                                                                 \\
                               & \multicolumn{1}{l|}{}   & 1     & 2     & 3     & 4     & 5     & 6     & 7     & 8     & 9     & 10    & 11    & 12    & 13    & 14    & 15    & 16    & 17    & 18    & 19    & 20    \\ \cline{2-22} 
\multirow{25}{*}{\rotatebox{90}{Resource $i$}} & \multicolumn{1}{l|}{1}  & 1     & 0     & 0     & 1     & 0     & 1     & 0     & 0     & 0     & 0     & 1     & 0     & 1     & 1     & 0     & 1     & 1     & 0     & 0     & 1     \\
                               & \multicolumn{1}{l|}{2}  & 0     & 0     & 0     & 0     & 1     & 1     & 1     & 1     & 0     & 0     & 0     & 0     & 0     & 0     & 0     & 1     & 0     & 1     & 0     & 0     \\
                               & \multicolumn{1}{l|}{3}  & 0     & 1     & 0     & 1     & 0     & 1     & 1     & 1     & 0     & 0     & 1     & 1     & 1     & 1     & 1     & 0     & 1     & 1     & 1     & 1     \\
                               & \multicolumn{1}{l|}{4}  & 1     & 0     & 1     & 1     & 1     & 0     & 0     & 1     & 0     & 0     & 0     & 0     & 1     & 0     & 1     & 1     & 0     & 0     & 1     & 1     \\
                               & \multicolumn{1}{l|}{5}  & 1     & 1     & 1     & 1     & 0     & 0     & 0     & 1     & 1     & 0     & 0     & 0     & 0     & 0     & 0     & 1     & 1     & 1     & 0     & 0     \\
                               & \multicolumn{1}{l|}{6}  & 0     & 1     & 0     & 0     & 1     & 0     & 0     & 0     & 1     & 1     & 0     & 0     & 0     & 0     & 1     & 0     & 0     & 0     & 0     & 1     \\
                               & \multicolumn{1}{l|}{7}  & 0     & 0     & 0     & 1     & 1     & 1     & 1     & 1     & 1     & 0     & 0     & 1     & 1     & 0     & 1     & 1     & 1     & 1     & 0     & 1     \\
                               & \multicolumn{1}{l|}{8}  & 0     & 1     & 0     & 0     & 0     & 0     & 1     & 1     & 1     & 1     & 1     & 0     & 0     & 0     & 1     & 1     & 1     & 0     & 0     & 1     \\
                               & \multicolumn{1}{l|}{9}  & 1     & 0     & 0     & 1     & 0     & 0     & 1     & 1     & 0     & 1     & 1     & 1     & 0     & 1     & 1     & 0     & 0     & 0     & 0     & 0     \\
                               & \multicolumn{1}{l|}{10} & 1     & 0     & 0     & 0     & 0     & 0     & 0     & 0     & 0     & 0     & 1     & 0     & 1     & 1     & 0     & 1     & 0     & 0     & 1     & 0     \\
                               & \multicolumn{1}{l|}{11} & 1     & 0     & 1     & 1     & 0     & 0     & 0     & 1     & 0     & 0     & 1     & 1     & 1     & 1     & 0     & 0     & 1     & 0     & 1     & 1     \\
                               & \multicolumn{1}{l|}{12} & 1     & 0     & 1     & 0     & 1     & 1     & 0     & 0     & 1     & 1     & 0     & 1     & 1     & 0     & 1     & 0     & 1     & 0     & 0     & 0     \\
                               & \multicolumn{1}{l|}{13} & 1     & 1     & 1     & 0     & 0     & 0     & 1     & 1     & 1     & 0     & 1     & 1     & 0     & 1     & 1     & 1     & 0     & 0     & 0     & 0     \\
                               & \multicolumn{1}{l|}{14} & 1     & 0     & 1     & 0     & 0     & 0     & 1     & 0     & 0     & 1     & 1     & 1     & 0     & 0     & 1     & 0     & 1     & 0     & 0     & 1     \\
                               & \multicolumn{1}{l|}{15} & 1     & 1     & 0     & 1     & 1     & 0     & 1     & 0     & 0     & 1     & 0     & 1     & 1     & 1     & 0     & 0     & 1     & 1     & 1     & 0     \\
                               & \multicolumn{1}{l|}{16} & 0     & 1     & 0     & 1     & 0     & 0     & 0     & 0     & 0     & 0     & 1     & 0     & 1     & 0     & 0     & 0     & 0     & 1     & 0     & 0     \\
                               & \multicolumn{1}{l|}{17} & 1     & 0     & 1     & 1     & 1     & 0     & 0     & 0     & 0     & 0     & 0     & 1     & 0     & 0     & 1     & 0     & 0     & 0     & 0     & 0     \\
                               & \multicolumn{1}{l|}{18} & 0     & 1     & 1     & 1     & 1     & 1     & 0     & 0     & 0     & 0     & 1     & 1     & 0     & 1     & 1     & 1     & 1     & 0     & 1     & 0     \\
                               & \multicolumn{1}{l|}{19} & 0     & 1     & 1     & 1     & 1     & 1     & 1     & 1     & 0     & 0     & 1     & 0     & 0     & 1     & 1     & 0     & 0     & 0     & 1     & 0     \\
                               & \multicolumn{1}{l|}{20} & 0     & 0     & 1     & 0     & 1     & 1     & 0     & 0     & 1     & 0     & 1     & 0     & 1     & 1     & 0     & 1     & 0     & 1     & 1     & 1     \\
                               & \multicolumn{1}{l|}{21} & 0     & 1     & 0     & 1     & 0     & 0     & 1     & 0     & 1     & 1     & 0     & 1     & 1     & 1     & 1     & 1     & 0     & 0     & 1     & 0     \\
                               & \multicolumn{1}{l|}{22} & 1     & 1     & 1     & 0     & 1     & 1     & 1     & 0     & 0     & 1     & 0     & 0     & 0     & 1     & 0     & 0     & 1     & 1     & 1     & 0     \\
                               & \multicolumn{1}{l|}{23} & 0     & 0     & 0     & 0     & 1     & 1     & 0     & 0     & 1     & 1     & 1     & 0     & 1     & 0     & 1     & 1     & 1     & 0     & 0     & 0     \\
                               & \multicolumn{1}{l|}{24} & 0     & 1     & 1     & 1     & 0     & 1     & 1     & 1     & 1     & 0     & 0     & 1     & 1     & 1     & 0     & 1     & 1     & 1     & 1     & 1     \\
                               & \multicolumn{1}{l|}{25} & 0     & 0     & 0     & 1     & 0     & 0     & 0     & 0     & 1     & 0     & 1     & 0     & 1     & 1     & 1     & 1     & 0     & 1     & 1     & 0     \\ \hline
                               & $p_j$                   & 0.094 & 0.047 & 0.011 & 0.047 & 0.082 & 0.011 & 0.011 & 0.058 & 0.105 & 0.07  & 0.094 & 0.011 & 0.011 & 0.058 & 0.023 & 0.07  & 0.058 & 0.058 & 0.07  & 0.011 \\
                               & $f_{j1}$                & 18    & 4     & 5     & 13    & 4     & 3     & 20    & 15    & 17    & 16    & 20    & 12    & 16    & 13    & 8     & 5     & 10    & 13    & 20    & 16    \\
                               & $f_{j2}$                & 4     & 2     & 3     & 6     & 3     & 2     & 16    & 14    & 14    & 2     & 2     & 4     & 2     & 3     & 4     & 4     & 2     & 7     & 6     & 3     \\
                               & $f_{j3}$                & 1     & 1     & 2     & 2     & 2     & 1     & 1     & 7     & 6     & 1     & 1     & 1     & 1     & 1     & 1     & 3     & 1     & 1     & 2     & 2     \\
                               & $q_{j1}$                & 0.108 & 0.116 & 0.025 & 0.062 & 0.162 & 0.069 & 0.305 & 0.169 & 0.016 & 0.129 & 0.197 & 0.496 & 0.009 & 0.114 & 0.023 & 0.171 & 0.056 & 0.104 & 0.202 & 0.22  \\
                               & $q_{j2}$                & 0.329 & 0.21  & 0.495 & 0.233 & 0.229 & 0.223 & 0.458 & 0.33  & 0.191 & 0.215 & 0.579 & 0.966 & 0.046 & 0.154 & 0.105 & 0.648 & 0.291 & 0.137 & 0.618 & 0.27  \\
                               & $q_{j3}$                & 0.408 & 0.335 & 0.585 & 0.619 & 0.396 & 0.281 & 0.764 & 0.44  & 0.452 & 0.563 & 0.62  & 0.993 & 0.269 & 0.682 & 0.26  & 0.852 & 0.723 & 0.993 & 0.802 & 0.588
\end{tabular}
\end{adjustbox}
\end{center}

\end{APPENDICES}

\end{document}

%% file: probing_states.tex
\begin{tikzpicture}[scale=0.25]
\tikzstyle{every node}+=[inner sep=0pt]

\draw (4.5,-23) node {$b_h,b_p,\varnothing$};

\draw (5.9,-40) node {$t$};
\draw (27.9,-40) node {$t-1/2$};
\draw (54.4,-40) node {$t-1$};

\draw (30,-9) node[text width=2cm] {$b_h-1,b_p,\texttt{a}$};
\draw (30,-23) node[text width=2cm] {$b_h,b_p-1,\texttt{p}$};
\draw (30,-34.1) node[text width=2cm] {$b_h,b_p,\texttt{r}$};

\draw (57,-9) node[text width=2cm] {$b_h-1,b_p,\varnothing$};
\draw (58,-17.2) node[text width=2.5cm] {$b_h-1,b_p-1,\varnothing$};
\draw (57,-34.1) node[text width=2cm] {$b_h,b_p,,\varnothing$};
\draw (57,-25.5) node[text width=2cm] {$b_h,b_p-1,\varnothing$};

\draw [black] (8.28,-22.07) -- (25.52,-10.83);
\fill [black] (25.52,-10.83) -- (24.58,-10.92) -- (25.19,-11.71);
\draw (18,-17) node [below] {$0$};
\draw [black] (8.87,-23.48) -- (24.93,-23.22);
\fill [black] (24.93,-23.22) -- (24.07,-22.84) -- (24.21,-23.83);
\draw (18,-24.3) node [below] {$0$};
\draw [black] (8.71,-24.95) -- (25.09,-33.05);
\fill [black] (25.09,-33.05) -- (24.51,-32.3) -- (24.16,-33.24);
\draw (18,-30.52) node [below] {$0$};

\draw [black] (35,-9) -- (51.4,-9);
\fill [black] (51.4,-9) -- (50.6,-8.5) -- (50.6,-9.5);
\draw (43,-9.5) node [below] {$R_j$};
\draw [black] (35,-22.18) -- (51.46,-17.82);
\fill [black] (51.46,-17.82) -- (50.58,-17.5) -- (50.79,-18.47);
\draw (43,-20.58) node [below] {$R_j$};
\draw [black] (35,-23.07) -- (51.41,-24.93);
\fill [black] (51.41,-24.93) -- (50.66,-24.36) -- (50.57,-25.36);
\draw (43,-24.83) node [below] {$0$};
\draw [black] (35,-34.1) -- (51.4,-34.1);
\fill [black] (51.4,-34.1) -- (50.6,-33.6) -- (50.6,-34.6);
\draw (43,-34.6) node [below] {$0$};
\end{tikzpicture}